\documentclass[12pt]{amsart}

\usepackage{a4wide,graphicx}
\usepackage[tight]{subfigure}
\usepackage{color}

\let\pa\partial  
\let\na\nabla  
  
\newcommand{\N}{{\mathbb N}}  
\newcommand{\R}{{\mathbb R}} 
\newcommand{\diver}{\operatorname{div}}

\newcommand{\T}{\mathcal{T}}
\newcommand{\E}{\mathcal{E}}
\newcommand{\dd}{\text{\rm d}}
\newcommand{\m}{\text{\rm m}}
\newcommand{\D}{\mathcal{D}}

\newtheorem{theorem}{Theorem}   
   
\newtheorem{proposition}[theorem]{Proposition}

\begin{document}  

\title[A finite volume scheme for a Keller-Segel model]{A finite volume scheme for
a Keller-Segel model \\ with additional cross-diffusion}

\author{Marianne Bessemoulin-Chatard}
\address{Laboratoire de Math\'ematiques, UMR6620, Universit\'e Blaise Pascal, 
24 Avenue des Landais, BP 80026, 63177 Aubi\`ere cedex, France}
\email{marianne.chatard@math.univ-bpclermont.fr} 

\author{Ansgar J\"ungel}
\address{Institute for Analysis and Scientific Computing, Vienna University of  
	Technology, Wiedner Hauptstra\ss e 8-10, 1040 Wien, Austria}
\email{juengel@tuwien.ac.at} 

\date{\today}

\thanks{The work of the first author was partially supported by the European Research Council Starting Grant 2009, project 239983-NuSiKiMo, and by the PHC Amadeus grant 2012, project 27238TD. The last author acknowledges partial support from the Austrian Science Fund (FWF), grants P20214, P22108, and I395, and the Austrian-French Project of the Austrian Exchange Service (\"OAD). The authors would like to thank 
C.~Chainais-Hillairet and F.~Filbet for fruitful suggestions and comments on this work} 

\begin{abstract}
A finite volume scheme for the (Patlak-) Keller-Segel model in two space dimensions
with an additional cross-diffusion term in the elliptic equation for the chemical
signal is analyzed. The main feature of the model is that there exists a new entropy functional yielding gradient estimates for the cell density and chemical concentration. The main features of the numerical scheme are positivity preservation, mass conservation, entropy stability, and---under additional assumptions---entropy dissipation. The existence of a discrete solution and its numerical convergence to the continuous solution is proved. Furthermore, temporal decay rates for convergence of the discrete solution to the homogeneous steady state is shown using a new discrete logarithmic Sobolev inequality. Numerical examples point out that the solutions exhibit intermediate states and that there exist nonhomogeneous stationary solutions with a finite cell density peak at the domain boundary.
\end{abstract}



\maketitle


\section{Introduction}

Chemotaxis, the directed movement of cells in response to chemical gradients, plays an important role in many biological fields, such as embryogenesis, immunology, cancer growth, and wound healing \cite{HiPa09,Per07}. At the macroscopic level, chemotaxis models can be formulated in terms of the cell density $n(x,t)$ and the concentration of the chemical signal $S(x,t)$. A classical model to describe the time evolution of these two variables is the (Patlak-) Keller-Segel system, suggested by Patlak in 1953 \cite{Pat53} and Keller and Segel in 1970 \cite{KeSe70}. Assuming that the time scale of the chemical signal is much larger than that of the cell movement, the classical parabolic-elliptic Keller-Segel equations read as follows:
$$
  \pa_t n = \diver(\na n-n\na S), \quad 0 = \Delta S + \mu n - S\quad \mbox{in }
  \Omega,
$$
where $\Omega\subset\R^2$ is a bounded domain or $\Omega=\R^2$, with homogeneous Neumann boundary and initial conditions. The parameter $\mu>0$ is the secretion rate at which the chemical substance is emitted by the cells. The nonlinear term $n\na S$ models the cell movement towards higher concentrations of the chemical signal. 

This model exhibits the phenomenon of cell aggregation. The more cells are aggregated, the more the attracting chemical signal is produced  by the cells. This process is counterbalanced by cell diffusion, but if the cell density is sufficiently large, the nonlocal chemical interaction dominates diffusion and results in a blow-up of the cell density. In two space dimensions, the critical threshold for blow-up is given by $M=\int_\Omega n_0dx=4\pi$ if $\Omega$ is a bounded connected domain with $C^2$ boundary \cite{Nag01} and $M=8\pi$ in the radial and whole-space case \cite{BCC12,NSY97}. The existence and uniqueness of smooth global-in-time solutions in the subcritical case is proved for bounded domains in \cite{JaLu92} and in the whole space in \cite{BDP06}. In the critical case $M=8\pi$, a global whole-space solution exists, which becomes unbounded as $t\to\infty$ \cite{BCM08}. Furthermore, there exist radially symmetric initial data such that, in the supercritical case, the solution forms a $\delta$-singularity in finite time \cite{HeVa96}. 

Motivated by numerical and modeling issues, the question how blow up can be avoided has been investigated intensively the last years. It has been suggested to modify the chemotactic sensitivity (modeling, e.g., volume-filling effects), to allow for degenerate cell diffusion, or to include suitable growth-death terms. We refer to \cite{HiJu11} for references. Another idea is to introduce additional cell diffusion in the equation for the chemical concentration \cite{CHJ12,HiJu11}. This diffusion term avoids, even for arbitrarily small diffusion constants, the blow-up and leads to the global-in-time existence of weak solutions \cite{HiJu11}. The model, which is investigated in this paper, reads as follows:
\begin{equation}
  \pa_t n = \diver(\na n-n\na S), \quad
   0 = \Delta S + \delta\Delta n + \mu n-S, \quad x\in\Omega,\ t>0, \label{ks}
\end{equation}
where $\delta>0$ is the additional diffusion constant. We impose the homogeneous Neumann boundary and initial conditions
\begin{equation}\label{bic}
  \na n\cdot\nu = \na S\cdot\nu = 0\quad\mbox{on }\pa\Omega,\ t>0, \quad
  n(\cdot,0)=n_0\quad\mbox{in }\Omega.
\end{equation}
The advantage of the additional diffusion term is that blow-up of solutions is translated to large gradients which may help to determine the blow-up time numerically. Another advantage is that the enlarged system \eqref{ks} exhibits an interesting entropy structure (see below). 

At first sight, the additional term $\delta\Delta n$ seems to complicate the mathematical analysis. Indeed, the resulting diffusion matrix of the system is neither symmetric nor positive definite, and we cannot apply the maximum principle to the equation for the chemical signal anymore. It was shown in \cite{HiJu11} that all these difficulties can be resolved by the observation that the above system possesses a logarithmic entropy,
$$
  E(t) = \int_\Omega\big(n(\log n-1)+1\big)dx,
$$
which is dissipated according to
\begin{equation}\label{1.ed}
  \frac{dE}{dt} + \int_\Omega\left(4|\na\sqrt{n}|^2 + \frac{1}{\delta}|\na S|^2
  + \frac{1}{\delta}S^2\right)dx = \frac{\mu}{\delta}\int_\Omega nS dx.
\end{equation}
Suitable Gagliardo-Nirenberg inequalities applied to the right-hand side lead to gradient estimates for $\sqrt{n}$ and $S$, which are the starting point for the global existence and long-time analysis. 

In this paper, we aim at developing a finite volume scheme which preserves the entropy structure on the discrete level by generalizing the scheme proposed in \cite{Fil06}. In contrast to \cite{Fil06}, we are able to prove the existence of discrete solutions and their numerical convergence to the continuous solution for all values of the initial mass. Moreover, we show that the discrete solution converges for large times to the homogeneous steady state if $\mu$ or $1/\delta$ are sufficiently small, using a new discrete logarithmic Sobolev inequality (Proposition \ref{prop.dlsi}).

In the literature, there exist several approaches to solve the classical Keller-Segel system numerically. The parabolic-elliptic model was approximated by using finite difference \cite{SaSu05,TSL00} or finite element methods \cite{Mar03,Sai07,SSKT10}. Also a dynamic moving-mesh method \cite{BCR05}, a variational steepest descent approximation scheme \cite{BCC08}, and a stochastic particle approximation \cite{HaSc09,HaSc11} were developed. Concerning numerical schemes for the parabolic-parabolic model (in which $\pa_t S$ is added to the second equation in \eqref{ks}), we mention the second-order central-upwind finite volume method of \cite{ChKu08}, the discontinuous finite element approach of \cite{EpIz09}, and the conservative finite element scheme of \cite{Sai12}. We also cite the paper \cite{BCW10} for a mixed finite element discretization of a Keller-Segel model with nonlinear diffusion.

There are only a few works in which a numerical analysis of the scheme was performed. Filbet proved the existence and numerical convergence of finite volume solutions \cite{Fil06}. Error estimates for a conservative finite element approximation were shown by Saito \cite{Sai07,Sai12}. Epshteyn and Izmirlioglu proved error estimates for a fully discrete discontinuous finite element method \cite{EpIz09}. Convergence proofs for other schemes can be found in, e.g., \cite{BCC08,HaSc11}.

This paper contains the first numerical analysis for the Keller-Segel model \eqref{ks} with additional cross-diffusion. Its originality comes from the fact that we ``translate'' all the analytical properties of \cite{HiJu11} on a discrete level, namely positivity preservation, mass conservation, entropy stability, and entropy dissipation (under additional hypotheses). 

The paper is organized as follows. Section \ref{sec.main} is devoted to the description of the finite volume scheme and the formulation of the main results. The existence of a discrete solution is shown in Section \ref{sec.ex}. A discrete version of the entropy-dissipation relation \eqref{1.ed} and corresponding gradient estimates are proved in Section \ref{sec.est}. These estimates allow us to obtain in Section \ref{sec.conv} the convergence of the discrete solution to the continuous one when the approximation parameters tend to zero. A proof of the discrete logarithmic Sobolev inequality is given in Section \ref{sec.dlsi}.
The long-time behavior of the discrete solution is investigated in Section \ref{sec.long}. Finally, we present some numerical examples in Section \ref{sec.num} and compare the discrete solutions to our model \eqref{ks} with those computed from the classical Keller-Segel system.


\section{Numerical scheme and main results}\label{sec.main}

In this section, we introduce the  finite volume scheme and present our main results.

\subsection{Notations and assumptions}\label{sec.nota}

Let $\Omega\subset\R^2$ be an open, bounded, polygonal subset. An admissible mesh of $\Omega$ is given by a family $\T$ of control volumes (open and convex polygons), a family $\E$ of edges, and a family of points $(x_K)_{K\in\T}$ which satisfy Definition 9.1 in \cite{EGH00}. This definition implies that the straight line between two neighboring centers of cells $(x_K,x_L)$ is orthogonal to the edge $\sigma=K|L$. For instance, Voronoi meshes are admissible meshes \cite[Example 9.2]{EGH00}. Triangular meshes satisfy the admissibility condition if all angles of the triangles are smaller than $\pi/2$ \cite[Example 9.1]{EGH00}.

We distinguish the interior edges $\sigma\in\E_{\rm int}$ and the boundary edges $\sigma\in\E_{\rm ext}$. The set of edges $\E$ equals the union $\E_{\rm int}\cup\E_{\rm ext}$. For a control volume $K\in\T$, we denote by $\E_K$ the set of its edges, by $\E_{{\rm int},K}$ the set of its interior edges, and by $\E_{{\rm ext},K}$ the set of edges of $K$ included in $\pa\Omega$. 

Furthermore, we denote by d the distance in $\R^2$ and by m the Lebesgue measure in $\R^2$ or $\R$. We assume that the family of meshes satisfies the following regularity requirement: there exists $\xi>0$ such that for all $K\in\T$ and all $\sigma\in\E_{{\rm int},K}$ with $\sigma=K|L$, it holds
\begin{equation}\label{reg.mesh}
  \dd(x_K,\sigma)\ge \xi \,\dd(x_K,x_L).
\end{equation}
This hypothesis is needed to apply discrete Sobolev-type inequalities \cite{BCF12}. Introducing for $\sigma\in\E$ the notation
$$
  d_\sigma = \left\{\begin{array}{ll}
  \dd(x_K,x_L) &\quad\mbox{if }\sigma\in\E_{\rm int},\ \sigma=K|L, \\
  \dd(x_K,\sigma) &\quad\mbox{if }\sigma\in\E_{{\rm ext},K},
  \end{array}\right.
$$
we define the transmissibility coefficient 
$$
  \tau_\sigma = \frac{\m(\sigma)}{d_\sigma}, \quad \sigma\in\E.
$$
The size of the mesh is denoted by 
$$
  \Delta x = \max_{K\in\T}\text{diam}(K).
$$
Let $T>0$ be some final time and $M_T$ the number of time steps. Then the time step size and the time points are given by, respectively,
$$
  \Delta t = \frac{T}{M_T}, \quad t^k = k\Delta t, \quad 0\le k\le M_T.
$$
We denote by $ \mathcal{D}$ an admissible space-time discretization of $ \Omega_{T}=\Omega \times (0,T)$ composed of an admissible mesh $\T$ of $ \Omega$ and the values $ \Delta t$ and $M_{T}$. The size of this space-time discretization $\mathcal{D}$ is defined by $ \eta=\max \{\Delta x, \Delta t\}$.

Let $X(\T)$ be the linear space of functions $\Omega\to\R$ which are constant on each cell $K\in\T$. We define on $X(\T)$ the discrete $L^p$ norm, discrete $W^{1,p}$ seminorm, and discrete $W^{1,p}$ norm by, respectively,
\begin{align*}
  \|u\|_{0,p,\T} &= \left(\sum_{K\in\T}\m(K)|u|^p\right)^{1/p}, \\ 
  |u|_{1,p,\T} &= \left(\sum_{\sigma\in\E_{\rm int}}\frac{\m(\sigma)}{d_\sigma^{p-1}}
  |D_\sigma u|^p\right)^{1/p}, 
\end{align*}
\begin{align*}
  \|u\|_{1,p,\T} &= \|u\|_{0,p,\T} + |u|_{1,p,\T},
\end{align*}
where $u\in X(\T)$, $1\le p<\infty$, and $D_\sigma u=|u_K-u_L|$ for $\sigma=K|L\in\E_{\rm int}$. 


\subsection{Finite volume scheme and main results}

We are now in the position to define the finite volume discretization of \eqref{ks}-\eqref{bic}. Let $ \mathcal{D}$ be a finite volume discretization of $\Omega_{T}$. The initial datum $n_0$ is approximated by its $L^2$ projection on control volumes:
\begin{equation}\label{fv.n0}
  n^0_\D = \sum_{K\in\T}n_K^0\mathbf{1}_K, \quad\mbox{where }
  n^0_K = \frac{1}{\m(K)}\int_K n_0(x)dx,
\end{equation}
and $\mathbf{1}_K$ is the characteristic function on $K$. Denoting by $n^k_K$ and $S^k_K$ approximations of the mean value of $n(\cdot,t^k)$ and $S(\cdot,t^k)$ on $K$, respectively, the numerical scheme reads as follows:
\begin{align}
  & \m(K)\frac{n^{k+1}_K-n^k_K}{\Delta t}
  - \sum_{\sigma\in\E_K}\tau_\sigma Dn_{K,\sigma}^{k+1}
  + \sum_{\substack{\sigma\in\E_{\rm int},\\ \sigma=K|L}}\tau_\sigma\big(
  (DS^{k+1}_{K,\sigma})^+ n_K^{k+1} - (DS^{k+1}_{K,\sigma})^- n_L^{k+1}\big) = 0, 
  \label{fv.n} \\
  & -\sum_{\sigma\in\E_K}\tau_\sigma DS^{k+1}_{K,\sigma}
  - \delta\sum_{\sigma\in\E_K}\tau_\sigma Dn^{k+1}_{K,\sigma}
  = \m(K)(\mu n_K^{k+1}-S_K^{k+1}), \label{fv.S}
\end{align}
for all $K\in\T$ and $0\le k\le M_T-1$. Here, $v^+=\max\{0,v\}$, $v^-=\max\{0,-v\}$, and 
\begin{equation}\label{def.D}
  DU_{K,\sigma}^k = \left\{\begin{array}{ll}
  U_L^k-U_K^k &\quad\mbox{for }\sigma=K|L\in\E_{{\rm int},K}, \\
  0 &\quad\mbox{for }\sigma\in\E_{{\rm ext},K}.
  \end{array}\right.
\end{equation}
The approximation $S^0_K$ is computed from \eqref{fv.S} with $k=-1$. This scheme is based on a fully implicit Euler discretization in time and a finite volume approach for the volume variable. The implicit scheme allows us to establish discrete entropy-dissipation estimates which would not be possible with an explicit scheme. This approximation is similar to that in \cite{Fil06} except the additional cross-diffusion term in the second equation.

The numerical approximations $n_\D$ and $S_\D$ of $n$ and $S$ are defined by
$$
  n_\D(x,t) = \sum_{K\in\T}n_K^{k+1}\mathbf{1}_K(x), \quad
  S_\D(x,t) = \sum_{K\in\T}S_K^{k+1}\mathbf{1}_K(x), \quad\mbox{where }
  x\in\Omega,\ t\in(t^k,t^{k+1}],
$$
and $k=0,\ldots,M_T-1$. Furthermore, we define approximations $\na^\D n_\D$ and $\na^\D S_\D$ of the gradients of $n$ and $S$, respectively. To this end, we introduce a dual mesh: for $K\in\T$ and $\sigma\in\E_K$, let $T_{K,\sigma}$ be defined by:
\begin{itemize}
\item If $\sigma=K|L\in\E_{{\rm int},K}$, $T_{K,\sigma}$ is the cell (``diamond'') whose vertices are given by $x_K$, $x_L$, and the end points of the edge $\sigma=K|L$.
\item If $\sigma\in\E_{{\rm ext},K}$, $T_{K,\sigma}$ is the cell (``triangle'') whose vertices are given by $x_K$ and the end points of the edge $\sigma=K|L$.
\end{itemize}
An example of construction of $T_{K ,\sigma}$ can be found in \cite{CLP03}. Clearly, $T_{K,\sigma}$ defines a partition of $\Omega$. The approximate gradient $\na^\D n_\D$ is a piecewise constant function, defined in $\Omega_T=\Omega\times(0,T)$ by
$$
  \na^\D n_\D(x,t) = \frac{\m(\sigma)}{\m(T_{K,\sigma})}Dn_{K,\sigma}^{k+1}
  \nu_{K,\sigma}, \quad x\in T_{K,\sigma},\ t\in(t^k,t^{k+1}),
$$
where $Dn_{K,\sigma}^{k+1}$ is given as in \eqref{def.D} and $\nu_{K,\sigma}$ is the unit vector normal to $\sigma$ and outward to $K$. The approximate gradient $\na^\D S_\D$ is defined in a similar way.

Our first result is the existence of solutions to the finite volume scheme.

\begin{theorem}[Existence of finite volume solutions]\label{thm.ex}
Let $\Omega\subset\R^2$ be an open, bounded, poly\-gonal subset and let $\D$ be an admissible discretization of $\Omega \times (0,T)$. The initial datum satisfies $n_0\in L^2(\Omega)$, $n_0\ge 0$ in $\Omega$. Then there exists a solution $\{(n_{K}^{k},S_{K}^{k}),\,K \in \T,\, 0 \leq k \leq M_{T}\} $ to \eqref{fv.n0}-\eqref{fv.S} satisfying
\begin{align}
  & n_K^k\ge 0 \quad\mbox{for all }K\in\T,\, k\ge 0, 
  \label{ex.1} \\
  & \sum_{K\in\T}\m(K)n_K^k = \sum_{K\in\T}\m(K)n_K^0 = \|n_0\|_{L^1(\Omega)} 
  \quad \mbox{for all } k\ge 0. \label{ex.2}
\end{align}
\end{theorem}

Properties \eqref{ex.1} and \eqref{ex.2} show that the scheme is positivity preserving and mass conser\-ving. It is also entropy stable; see \eqref{ent.stab} below.

Let $(\D_\eta)_{\eta>0}$ be a sequence of admissible space-time discretizations indexed by the size $\eta= \max\{\Delta x,\Delta t\}$ of the discretization. We denote by $(\T_{\eta})_{\eta>0}$ the corresponding meshes of $\Omega$. We suppose that these discretizations satisfy \eqref{reg.mesh} uniformly in $\eta$, i.e., $\xi>0$ does not depend on $\eta$. Let $(n_\eta,S_\eta):=(n_{\D_\eta},S_{\D_\eta})$ be a finite volume solution, constructed in Theorem \ref{thm.ex}, on the discretization $\D_\eta$. We set $\na^\eta:=\na^{\D_\eta}$. Our second result concerns the convergence of $(n_\eta,S_\eta)$ to a weak solution $(n,S)$ to \eqref{ks}-\eqref{bic}.

\begin{theorem}[Convergence of the finite volume solutions]\label{thm.conv}
Let the assumptions of Theorem \ref{thm.ex} hold. Furthermore, let $(\D_\eta)_{\eta>0}$ be a sequence of admissible discretizations satisfying \eqref{reg.mesh} uniformly in $\eta$, and let $(n_\eta,S_\eta)$ be a sequence of finite volume solutions to \eqref{fv.n0}-\eqref{fv.S}. Then there exists $(n,S)$ such that, up to a subsequence,
\begin{align*}
  n_\eta \to n &\quad\mbox{strongly in }L^2(\Omega_T), \\
  \na^\eta n_\eta \rightharpoonup \na n &\quad\mbox{weakly in }L^2(\Omega_T), \\
  S_\eta\rightharpoonup S, \ \na^\eta S_\eta\rightharpoonup \na S
  &\quad\mbox{weakly in }L^2(\Omega_T),
\end{align*}
and $(n,S)\in L^2(0,T;H^1(\Omega))^2$ is a weak solution to \eqref{ks}-\eqref{bic} in the sense of
\begin{align}
  \int_0^T\int_\Omega(n\pa_t\phi - \na n\cdot\na\phi + n\na S\cdot\na\phi)dx
  + \int_\Omega n_0\phi(\cdot,0)dx &= 0, \label{weak.n} \\
  \int_0^T\int_\Omega(-\na S\cdot\na\phi - \delta\na n\cdot\na\phi + \mu n\phi-S\phi)
  dx &= 0 \label{weak.S}
\end{align}
for all test functions $\phi\in C_0^\infty(\Omega\times[0,T])$.
\end{theorem}

It is shown in \cite[Theorem 1.3]{HiJu11} that, if the secretion rate $\mu>0$ is sufficiently small or the diffusion parameter $\delta>0$ is sufficiently large, the solution $(n,S)$ to \eqref{ks}-\eqref{bic} converges exponentially fast to the homogeneous steady state $(n^*,S^*)$, where $n^*=\|n_0\|_{L^1(\Omega)}/\m(\Omega)$ and $S^*=\mu n^*$. The proof in \cite{HiJu11} is based on the logarithmic Sobolev 
inequality. Therefore, we state first a novel discrete version of this inequality,
which is proved in Section \ref{sec.dlsi}.

\begin{proposition}[Discrete logarithmic Sobolev inequality]\label{prop.dlsi}
Let $\Omega\subset\R^d$ $(d\ge 1)$ be an open bounded polyhedral domain and let $\T$ be
an admissible mesh of $\Omega$ satisfying \eqref{reg.mesh}. Then there
exists a constant $C_L>0$ only depending on $\Omega$, $d$, and $\xi$ such that
for all $u\in X(\T)$,
$$
  \int_\Omega u^2\log\frac{u^2}{m^{-1}\|u\|_{0,2,\T}^2}dx
  \le C_L|u|_{1,2,\T}^2,
$$
where we abbreviated $m=\m(\Omega)$.
\end{proposition}

The constant $C_L>0$ can be made more precise. Let $C_S(q)>0$ be the constant in
the discrete Sobolev inequality \cite[Theorem 4]{BCF12}
\begin{equation}\label{dsi}
  \|u\|_{0,q,\T} \le \frac{C_S(q)}{\sqrt{\xi}}\|u\|_{1,2,\T}\quad\mbox{for }
  u\in X(\T),
\end{equation}
where $1\le q\le 2d/(d-2)$ (and $1\le q<\infty$ if $d\le 2$), and
let $C_P(r)>0$ be the constant in the discrete Poincar\'e-Wirtinger inequality \cite[Prop.~1]{BCF12}
\begin{equation}\label{dpi}
  \|u-\bar u\|_{0,r,\T} \le \frac{C_P(r)}{\sqrt{\xi}}|u|_{1,2,\T}
  \quad\mbox{for } u\in X(\T),
\end{equation}
where $1\le r\le 2$ and $\bar u=m^{-1}\int_\Omega u(x)dx$, $m=\m(\Omega)$. Then
$$
  C_L = \frac{q}{(q-2)\xi}\left(C_S(q)^2 + \frac{C_S(q)^2C_P(2)^2}{\xi} 
  + \frac{q-4}{q}C_P(2)^2\right).
$$

For our result on the long-time behavior, we introduce the discrete relative entropy
$$
  E[n^k|n^*] = \sum_{K\in\T}\m(K)n_K^k\log\left(\frac{n_K^k}{n^*}\right)\ge 0, \quad
  k\ge 0.
$$

\begin{theorem}[Long-time behavior of finite volume solutions]\label{thm.long}
Let the assumptions of Theorem \ref{thm.ex} hold and let $(n_\D,S_\D)$ be a solution to \eqref{fv.n0}-\eqref{fv.S}. Then for all $k\ge 0$,
$$
  E[n^{k+1}|n^*] + \Delta t(1-C^*)\big|\sqrt{n^{k+1}}\big|_{1,2,\T}^2
  + \frac{\Delta t}{2\delta}\|S^{k+1}-S^*\|_{1,2,\T}^2 \le E[n^k|n^*],
$$
where $C^*=\mu^2 C(\Omega)^2\|n_0\|_{L^1(\Omega)}/(\delta\xi)$, $C(\Omega)>0$ only depending on $\Omega$, and $\xi$ is the parameter in \eqref{reg.mesh}. In particular, if $C^* < 1$, the discrete relative entropy is nonincreasing and 
\begin{align*}
  \|n^k-n^*\|_{0,1,\T} &\le \sqrt{4E[n^0|n^{*}]\|n_0\|_{L^1(\Omega)}}
  \left(1+\frac{1-C^*}{C_L}\Delta t \right)^{-k/2}, \\
  \|S^k-S^*\|_{1,2,\T}
  &\le \sqrt{\frac{2\delta E[n^0|n^{*}]}{\Delta t}}
  \left(1+\frac{1-C^*}{C_L}\Delta t
  \right)^{-(k-1)/2}, \quad k\in\N,
\end{align*}
where $C_L>0$ is the constant in the discrete logarithmic Sobolev inequality
(see Proposition \ref{prop.dlsi}). In particular, for each $K\in\T$,
$$
  (n_K^k,S_K^k)\to (n^*,S^*)\quad\mbox{as }k\to\infty.
$$
\end{theorem}

In \cite{HiJu11}, the long-time behavior of solutions is shown under the condition $\mu^2\|n_0\|_{L^1(\Omega)}/\delta < C(\Omega)$ for a constant $C(\Omega)>0$ appearing in some Poincar\'e-Sobolev inequality. We observe that our condition depends additionally on the regularity of the finite volume mesh. The convergence
of $\|n^k-n^*\|_{0,1,\T}$ is approximately exponential for sufficiently
small $\Delta t>0$. The convergence result for $\|S^k-S^*\|_{1,2,\T}$ is weaker
in view of the factor $(\Delta t)^{-1}$. In \cite{HiJu11}, the exponential
convergence of $S-S^*$ is shown in the $L^2(\Omega)$ norm only and therefore,
our result is not surprising. 


\section{Existence of finite volume solutions}\label{sec.ex}

In this section, we prove Theorem \ref{thm.ex}. The proof is based on the Brouwer fixed-point theorem. Let $k\in\{1,\ldots,M_T-1\}$ and let $(n^k,S^k)$ be a solution to \eqref{fv.n} and \eqref{fv.S}, with $k+1$ replaced by $k$, satisfying \eqref{ex.1}-\eqref{ex.2}. We introduce the set
$$
  Z = \{u\in X(\T):u\ge 0\mbox{ in }\Omega,\ 
  \|u\|_{L^1(\Omega)}\le \|n_0\|_{L^1(\Omega)}\}.
$$
The finite-dimensional space $Z$ is convex and compact. In the following, we define the fixed-point operator by solving a linearized problem. First, we construct $\widetilde S\in X(\T)$ using the following scheme:
\begin{equation}\label{tildes}
  -\sum_{\sigma\in\E_K}\tau_\sigma D\widetilde S_{K,\sigma} + \m(K)\widetilde S_K
  = \delta\sum_{\sigma\in\E_K}\tau_\sigma  Dn_{K,\sigma}^k + \mu\m(K)n_K^k, \quad
  K\in\T.
\end{equation}
Second, we compute $\widetilde n\in X(\T)$ using the scheme
\begin{equation}\label{tilden}
  \frac{\m(K)}{\Delta t}(\widetilde n_K-n_K^k)
  - \sum_{\sigma\in\E_K}\tau_\sigma D\widetilde n_{K,\sigma}
  + \sum_{\substack{\sigma\in\E_{\rm int},\\ \sigma=K|L}}\tau_\sigma
  \big((D\widetilde S_{K,\sigma})^+ \widetilde n_K - (D\widetilde S_{K,\sigma})^- 
  \widetilde n_L\big) = 0.
\end{equation}

{\em Step 1: Existence and uniqueness for \eqref{tildes} and \eqref{tilden}.} The linear system \eqref{tildes} can be written as $A\widetilde S=b$, where $A$ is the matrix with the elements
$$
  A_{K,K} = \sum_{\sigma\in\E_K}\tau_\sigma + \m(K), \quad
  A_{K,L} = -\tau_\sigma\quad\mbox{ for }K,L\in\T\mbox{ with }
  \sigma=K|L\in\E_{{\rm int},K},
$$
and $b$ is the vector with the elements
$$
  b_K = \delta\sum_{\sigma\in\E_K}\tau_\sigma Dn_{K,\sigma}^k + \mu\m(K)n_K^k.
$$
Since for all $L\in\T$,
$$
  |A_{L,L}| - \sum_{K\neq L}|A_{K,L}| = \m(L) > 0,
$$
the matrix $A$ is strictly diagonally dominant with respect to the columns and hence, $A$ is invertible. This shows the unique solvability of \eqref{tildes}.

Similarly, \eqref{tilden} can be written as $B\widetilde n=c$, where $B$ is the matrix with the elements
\begin{align*}
  B_{K,K} &= \frac{\m(K)}{\Delta t} + \sum_{\sigma\in\E_K}\tau_\sigma
  (1+(D\widetilde S_{K,\sigma})^+), \quad K\in\T, \\
  B_{K,L} &= -\tau_\sigma(1+(D\widetilde S_{K,\sigma})^-), 
  \quad\mbox{ for }K,L\in\T\mbox{ with }\sigma=K|L\in\E_{{\rm int},K},
\end{align*}
and $c$ is the vector with the elements $c_K=\m(K)n^k_K/\Delta t$, $K\in\T$. The diagonal elements of $B$ are positive, and the off-diagonal elements are nonpositive. Moreover, $B$ is strictly diagonal dominant with respect to the columns since for all $\sigma=K|L\in\E_{{\rm int},K}$, we have $D\widetilde S_{L,\sigma}=-D\widetilde S_{K,\sigma}$ which yields $(D\widetilde S_{L,\sigma})^+=(D\widetilde S_{K,\sigma})^-$ and hence,
$$
  |B_{L,L}| - \sum_{K\neq L}|B_{K,L}| = \frac{\m(L)}{\Delta t} > 0.
$$
We infer that $B$ is an M-matrix and invertible, which gives the existence and uniqueness of a solution $\widetilde n$ to \eqref{tilden}.

The M-matrix property of $B$ implies that $B^{-1}$ is positive. As a consequence, since $n^k$ and $c$ are nonnegative componentwise, by the induction hypothesis, $\widetilde n=B^{-1}c$ is nonnegative componentwise. This means that $\widetilde n$ satisfies \eqref{ex.1}. Summing \eqref{tilden} over $K\in\T$, we compute
$$
  \sum_{K\in\T}\m(K)\widetilde n_K = \sum_{K\in\T}\m(K)n_K^k 
  = \|n_0\|_{L^1(\Omega)}.
$$

{\em Step 2: Continuity of the fixed-point operator.} The solution to \eqref{tildes} and \eqref{tilden} defines the fixed-point 
operator $F:Z\to Z$, $F(n)=\widetilde n$. We have to show that $F$ is continuous. For this, let $(n^\gamma)_{\gamma\in\N}\subset Z$ be a sequence converging to $n$ in $X(\T)$ as $\gamma\to\infty$. Setting $\widetilde n^\gamma=F(n^\gamma)$ and $\widetilde n=F(n)$, we have to prove that $\widetilde n^\gamma\to \widetilde n$ in $X(\T)$. Using the scheme \eqref{tildes}, we construct first $\widetilde S^\gamma$ (respectively, $\widetilde S$) from  $n^\gamma$ (respectively, $n$). Then, using the scheme \eqref{tilden}, we obtain $\widetilde n^\gamma$ (respectively, $\widetilde n$).

We claim that $\widetilde S^\gamma-\widetilde S\to 0$ in $X(\T)$ as $\gamma\to\infty$. Indeed, since the map $n \mapsto \tilde{S}$, where $ \tilde{S}$ is constructed from \eqref{tildes}, is linear on the finite dimensional space $X(\T)$, it is obviously continuous. 
Moreover, using the scheme \eqref{tilden} and performing the same computations as in the proof of Theorem 2.1 in \cite{Fil06}, it follows that
$$
  \sum_{K\in\T}\m(K)|\widetilde n^\gamma_K-\widetilde n_K|
  \le 2\Delta t\bigg(\sum_{K\in\T}|\widetilde n_K|^2\bigg)^{1/2}
  \bigg(\sum_{K\in\T}\sum_{\sigma\in\E_K}\tau_\sigma
  |D(\widetilde S^\gamma-\widetilde S)_{K,\sigma}|^2\bigg)^{1/2}.
$$
The right-hand side converges to zero as $\widetilde S^\gamma\to\widetilde S$ in $X(\T)$, which proves that $\widetilde n^\gamma \to \widetilde n$ in $X(\T)$.

{\em Step 3: Application of the fixed-point theorem.} The assumptions of the Brouwer fixed-point theorem are satisfied, implying the existence of a fixed point of $F$, i.e.\ of a solution $n^{k+1}$ to \eqref{fv.n} satisfying \eqref{ex.1}. We have shown in Step 1 that \eqref{ex.2} holds for $n^{k+1}$. Finally, we construct $S^{k+1}$ using scheme \eqref{fv.S}.


\section{A priori estimates}\label{sec.est}

The proof of Theorem \ref{thm.conv} is based on suitable a priori estimates which are shown in this section. We introduce a discrete version of the entropy functional used in \cite{HiJu11}:
$$
  E^k = \sum_{K\in\T}\m(K)H(n_K^k), \quad\mbox{where }H(s)=s(\log s-1)+1.
$$

\begin{proposition}[Entropy stability]\label{prop.ent}
There exists a constant $C>0$ only depending on $\Omega$, $\mu$, $\delta$, $\|n_0\|_{L^1(\Omega)}$, and $\xi$ (see \eqref{reg.mesh}) such that for all $k\ge 0$,
\begin{align}
  E^{k+1}-E^k &+ \frac{\Delta t}{2}\sum_{\sigma\in\E_{\rm int}}\tau_\sigma
  \big|D(\sqrt{n^{k+1}})_{K,\sigma}\big|^2 + \frac{\Delta t}{\delta}
  \sum_{K\in\T}\m(K)|S^{k+1}|^2 \nonumber \\
  &{}+ \frac{\Delta t}{\delta}
  \sum_{\sigma\in\E_{\rm int}}\tau_\sigma|DS_{K,\sigma}^{k+1}|^2 \le C\Delta t.
  \label{ent.stab}
\end{align}
\end{proposition}

\begin{proof}
By the convexity of $H$, we find that
$$
  E^{k+1}-E^k = \sum_{K\in\T}\m(K) \big(H(n_K^{k+1})-H(n_K^k)\big)
  \le \sum_{K\in\T}\m(K)\log(n_K^{k+1})(n_K^{k+1}-n_K^k).
$$
Inserting the scheme \eqref{fv.n}, we can write
\begin{align}
  E^{k+1}-E^k &\le \Delta t\sum_{\substack{\sigma\in\E_{\rm int},\\ \sigma=K|L}}\tau_\sigma
  (n^{k+1}_L-n^{k+1}_K)\log n_K^{k+1} \nonumber \\
  &\phantom{xx}{}
  - \Delta t\sum_{\substack{\sigma\in\E_{\rm int},\\ \sigma=K|L}}\tau_\sigma
  \big((DS^{k+1}_{K,\sigma})^+ n_K^{k+1} - (DS^{k+1}_{K,\sigma})^- n_L^{k+1}\big)
  \log n_K^{k+1} 
  =: I_1 + I_2. \label{aux1}
\end{align}
Now, we argue similarly as in the proof of Lemma 3.1 in \cite{Fil06}. We employ the symmetry of $\tau_\sigma$ and a Taylor expansion of $\log$ around $n_K^{k+1}$ to infer that
\begin{align*}
  I_1 &= -\Delta t\sum_{\substack{\sigma\in\E_{\rm int},\\ \sigma=K|L}}
  \tau_\sigma(n_K^{k+1}-n_L^{k+1})(\log n_K^{k+1}-\log n_L^{k+1}) \\
  &= -\Delta t\sum_{\substack{\sigma\in\E_{\rm int},\\ \sigma=K|L}}
  \tau_\sigma(n_K^{k+1}-n_L^{k+1})^2 \frac{1}{\bar n_\sigma^{k+1}} 
  = -\Delta t\sum_{\substack{\sigma\in\E_{\rm int},\\ \sigma=K|L}}
  \tau_\sigma \left(\frac{Dn_{K,\sigma}^{k+1}}{\sqrt{\bar n_\sigma^{k+1}}}\right)^2,
\end{align*}
where $\bar n_\sigma^{k+1}=t_\sigma n_K^{k+1}+(1-t_\sigma)n_L^{k+1}$ for some $t_\sigma\in(0,1)$. We perform a summation by parts in $I_2$, using again the symmetry of $\tau_\sigma$:
$$
  I_2 = -\Delta t\sum_{\substack{\sigma\in\E_{\rm int},\\ \sigma=K|L}}
  \tau_\sigma
  \big((DS^{k+1}_{K,\sigma})^+ n_K^{k+1} - (DS^{k+1}_{K,\sigma})^- n_L^{k+1}\big)
  (\log n_K^{k+1}-\log n_L^{k+1}).
$$
Reordering the sum and using the expression for $\bar n_\sigma^{k+1}$ in the Taylor expansion of log, it is shown in \cite[p.~468]{Fil06} that
$$
  I_2 \le -\Delta t
  \sum_{\substack{\sigma\in\E_{\rm int},\\ \sigma=K|L}}\tau_\sigma
  \bar n_\sigma^{k+1} DS_{K,\sigma}^{k+1}(\log n_K^{k+1}-\log n_L^{k+1}).
$$
The Taylor expansion shows that $\bar n_\sigma^{k+1}(\log n_K^{k+1}-\log n_L^{k+1})= n_K^{k+1}-n_L^{k+1}$, which gives
$$
  I_2 \le -\Delta t\sum_{\substack{\sigma\in\E_{\rm int},\\ \sigma=K|L}}\tau_\sigma
  DS_{K,\sigma}^{k+1}(n_K^{k+1}-n_L^{k+1}) 
  = \Delta t\sum_{\substack{\sigma\in\E_{\rm int},\\ \sigma=K|L}}\tau_\sigma
  DS_{K,\sigma}^{k+1}Dn_{K,\sigma}^{k+1}.
$$
Summarizing the estimates for $I_1$ and $I_2$, \eqref{aux1} leads to
\begin{equation}\label{aux2}
  E^{k+1}-E^k 
  \le -\Delta t\sum_{\substack{\sigma\in\E_{\rm int},\\ \sigma=K|L}}
  \tau_\sigma \left(\frac{Dn_{K,\sigma}^{k+1}}{\sqrt{\bar n_\sigma^{k+1}}}\right)^2
  + \Delta t\sum_{\substack{\sigma\in\E_{\rm int},\\ \sigma=K|L}}\tau_\sigma
  DS_{K,\sigma}^{k+1}Dn_{K,\sigma}^{k+1}.
\end{equation}
The first term can be estimated for $\sigma=K|L$ as follows:
\begin{equation}\label{aux3}
  \frac{|Dn_{K,\sigma}^{k+1}|}{\sqrt{\bar n_\sigma^{k+1}}}
  = \frac{\sqrt{n_L^{k+1}}+\sqrt{n_K^{k+1}}}{\sqrt{\bar n_\sigma^{k+1}}}
  |D(\sqrt{n^{k+1}})_{K,\sigma}| \ge |D(\sqrt{n^{k+1}})_{K,\sigma}|.
\end{equation}
In order to bound the second term, we multiply the scheme \eqref{fv.S} by $(\Delta t/\delta)S_K^{k+1}$ and sum over $K\in\T$:
\begin{align*}
  0 &= \frac{\Delta t}{\delta}\sum_{K\in\T}\sum_{\sigma\in\E_K}\tau_\sigma
  DS_{K,\sigma}^{k+1} S_K^{k+1} 
  + \Delta t\sum_{K\in\T}\sum_{\sigma\in\E_K}\tau_\sigma
  Dn_{K,\sigma}^{k+1} S_K^{k+1} \\
  &\phantom{xx}{}+ \frac{\mu\Delta t}{\delta}\sum_{K\in\T}\m(K)
  n_K^{k+1}S_K^{k+1} - \frac{\Delta t}{\delta}\sum_{K\in\T}\m(K)
  |S_K^{k+1}|^2.
\end{align*}
By summation by parts, we find that
\begin{align}
  \Delta t\sum_{\substack{\sigma\in\E_{\rm int},\\ \sigma=K|L}}\tau_\sigma
  Dn_{K,\sigma}^{k+1}DS_{K,\sigma}^{k+1}
  &= -\frac{\Delta t}{\delta}\sum_{\substack{\sigma\in\E_{\rm int},\\ \sigma=K|L}}
  \tau_\sigma|DS_{K,\sigma}^{k+1}|^2 
  + \frac{\mu\Delta t}{\delta}\sum_{K\in\T}\m(K)n_K^{k+1}S_K^{k+1} \nonumber \\
  &\phantom{xx}{}-\frac{\Delta t}{\delta}\sum_{K\in\T}\m(K)|S_K^{k+1}|^2.
  \label{aux4}
\end{align}
Inserting \eqref{aux4} into \eqref{aux2} and employing \eqref{aux3}, it follows that
\begin{align}
  E^{k+1}-E^k &+ \Delta t
  \sum_{\substack{\sigma\in\E_{\rm int},\\ \sigma=K|L}}\tau_\sigma
  |D(\sqrt{n^{k+1}})_{K,\sigma}|^2 
  + \frac{\Delta t}{\delta}\sum_{K\in\T}\m(K)|S_K^{k+1}|^2 \nonumber \\
  &{}+ \frac{\Delta t}{\delta}\sum_{\substack{\sigma\in\E_{\rm int},\\ \sigma=K|L}}
  \tau_\sigma |DS_{K,\sigma}^{k+1}|^2 
  \le \frac{\mu\Delta t}{\delta}\sum_{K\in\T}\m(K)n_K^{k+1}S_K^{k+1}. \label{aux.e}
\end{align}
It remains to estimate the right-hand side. We follow the proof of Proposition 3.1 in \cite{HiJu11}. The H\"older inequality yields
\begin{equation}\label{aux5}
  \mu\sum_{K\in\T}\m(K)n_K^{k+1}S_K^{k+1}
  \le \mu\|n^{k+1}\|_{0,6/5,\T}\|S^{k+1}\|_{0,6,\T}.
\end{equation}
The discrete $L^6$ norm of $S^{k+1}$ can be bounded by the discrete $H^1$ norm 
using the discrete Sobolev inequality \eqref{dsi}
$$
  \|S^{k+1}\|_{0,6,\T} \le C\|S^{k+1}\|_{1,2,\T},
$$
where $C>0$ depends only on $\Omega$ and $\xi$. For the discrete $L^{6/5}$ norm of $n^{k+1}$, we employ the discrete Gagliardo-Nirenberg inequality \cite[Theorem 3]{BCF12} with $\theta=1/6$:
$$
  \|n^{k+1}\|_{0,6/5,\T} = \big\|\sqrt{n^{k+1}}\big\|_{0,12/5,\T}^2
  \le C\big\|\sqrt{n^{k+1}}\big\|_{0,2,\T}^{2(1-\theta)}
  \big\|\sqrt{n^{k+1}}\big\|_{1,2,\T}^{2\theta},
$$
where $C>0$ depends only on $\Omega$ and $\xi$. Mass conservation \eqref{ex.2} implies that
$$
  \|n^{k+1}\|_{0,6/5,\T} 
  \le C\|n_0\|_{L^1(\Omega)}^{1-\theta}\left(\|n_0\|_{L^1(\Omega)}^{1/2} 
  + \big|\sqrt{n^{k+1}}\big|_{1,2,\T}\right)^{2\theta}.
$$
With these estimates, \eqref{aux5} becomes
\begin{align*}
  \mu\sum_{K\in\T}\m(K)n_K^{k+1}S_K^{k+1}
  &\le C\mu\|S^{k+1}\|_{1,2,\T}\|n_0\|_{L^1(\Omega)}^{1-\theta}
  \left(\|n_0\|_{L^1(\Omega)}^{1/2} 
  + \big|\sqrt{n^{k+1}}\big|_{1,2,\T}\right)^{2\theta} \\
  &\le 2C\mu\|S^{k+1}\|_{1,2,\T}\|n_0\|_{L^1(\Omega)}^{1-\theta}
  \left(\|n_0\|_{L^1(\Omega)} 
  + \big|\sqrt{n^{k+1}}\big|_{1,2,\T}^2\right)^{\theta}.
\end{align*}
Then, by Young's inequality with $p_1=2$, $p_2=2/(1-2\theta)$, $p_3=1/\theta$,
\begin{align*}
  \mu\sum_{K\in\T}\m(K)n_K^{k+1}S_K^{k+1}
  &\le \frac12\|S^{k+1}\|_{1,2,\T}^2 
  + C(\delta,\mu)\|n_0\|_{L^1(\Omega)}^{2(1-\theta)/(1-2\theta)} \\
  &\phantom{xx}{}+ \frac{\delta}{2}\left(\|n_0\|_{L^1(\Omega)}
  + \big|\sqrt{n^{k+1}}\big|_{1,2,\T}^2\right) \\
  &\le \frac12\|S^{k+1}\|_{1,2,\T}^2 
  + \frac{\delta}{2}\big|\sqrt{n^{k+1}}\big|_{1,2,\T}^2 
  + C(\delta,\mu,\|n_0\|_{L^1(\Omega)}).
\end{align*}
Together with \eqref{aux.e}, this finishes the proof.
\end{proof}

Summing \eqref{ent.stab} over $k=0,\ldots,M_T-1$ and using the mass conservation \eqref{ex.2}, we conclude immediately the following $\eta$-uniform bounds for the family of solutions $(n_\eta,S_\eta)_{\eta>0}$ to \eqref{fv.n0}-\eqref{fv.S} with discretizations $\D_\eta$:
\begin{align}
  & (n_\eta),\ (n_\eta\log n_\eta)\mbox{ are bounded in }L^\infty(0,T;L^1(\Omega)), 
  \label{unif.n} \\
  & (\na^\eta n_\eta)\mbox{ is bounded in }L^2(\Omega_T), \label{unif.nan} \\
  & (S_\eta)\mbox{ is bounded in }L^2(0,T;H^1(\Omega)). \label{nas}
\end{align}

\begin{proposition}\label{prop.bound}
The family $(n_\eta)_{\eta>0}$ is bounded in $L^2(0,T;H^1(\Omega))$.
\end{proposition}

\begin{proof}
First, we claim that $(n_\eta)$ is bounded in $L^2(0,T;W^{1,1}(\Omega))$. To simplify the notation, we write $n_\eta^{k+1}:=n_\eta(\cdot,t^{k+1}) \in X(\T_{\eta})$. Applying the Cauchy-Schwarz inequality, we obtain
\begin{align*}
  |n_\eta^{k+1}|_{1,1,\T_{\eta}}
  &= \sum_{\substack{\sigma\in\E_{\rm int},\\ \sigma=K|L}}\m(\sigma)
  |n_L^{k+1}-n_K^{k+1}| \\
  &\le \sum_{K\in\T_\eta}\sum_{\sigma\in\E_K}\sqrt{\tau_\sigma}
  \big|D(\sqrt{n^{k+1}})_{K,\sigma}\big|\cdot\sqrt{\m(\sigma)d_\sigma}
  \sqrt{n_K^{k+1}} \\
  &\le \big|\sqrt{n_\eta^{k+1}}\big|_{1,2,T_{\eta}}
  \left(\sum_{K\in\T_\eta}\Big(\sum_{\sigma\in\E_K}\m(\sigma)d_\sigma\Big)
  n_K^{k+1}\right)^{1/2}.
\end{align*}
Observe that in two space dimensions,
$$
  \sum_{\sigma\in\E_K}\m(\sigma)\dd(x_K,\sigma) = 2\m(K),
$$
since the straight line between $x_K$ and $x_L$ is orthogonal to the edge $\sigma=K|L$. Using this property, the mesh regularity assumption \eqref{reg.mesh}, and the mass conservation \eqref{ex.2}, it follows that
\begin{align}
  |n_\eta^{k+1}|_{1,1,\T_{\eta}}
  &\le \left(\frac{2}{\xi}\right)^{1/2}\big|\sqrt{n_\eta^{k+1}}\big|_{1,2,\T_{\eta}}
  \left(\sum_{K\in\T_{\eta}}\m(K)n_K^{k+1}\right)^{1/2} \nonumber \\
  &= \left(\frac{2}{\xi}\right)^{1/2}\big|\sqrt{n_\eta^{k+1}}\big|_{1,2,\T_{\eta}}
  \|n_0\|_{L^1(\Omega)}^{1/2}. \label{w11}
\end{align}
In view of the entropy stability estimate \eqref{ent.stab}, we infer that $(n_\eta)$ is bounded in $L^2(0,T;$ $W^{1,1}(\Omega))$. Because of the discrete Sobolev inequality \cite[Theorem~4]{BCF12},
$$
  \|n_\eta^{k+1}\|_{0,2,\T_{\eta}} \le C\|n_\eta^{k+1}\|_{1,1,T_{\eta}},
$$
the family $(n_\eta)$ is bounded in $L^2(\Omega_T)$. 

In order to estimate the approximate gradient of $n_\eta$, we employ \eqref{aux2}. The last term in \eqref{aux2} is treated as follows. We multiply \eqref{fv.S} by $\Delta t \,n_K^{k+1}$, sum over $K\in\T$, and sum by parts:
\begin{align*}
  \Delta t\sum_{\substack{\sigma\in\E_{\rm int},\\ \sigma=K|L}} &
  \tau_\sigma DS_{K,\sigma}^{k+1}Dn_{K,\sigma}^{k+1}
  = -\delta\Delta t\sum_{\substack{\sigma\in\E_{\rm int},\\ \sigma=K|L}}
  \tau_\sigma |Dn_{K,\sigma}^{k+1}|^2 \\
  &\phantom{xx}{}+ \mu\Delta t\sum_{K\in\T_\eta}
  \m(K)|n_K^{k+1}|^2 
  - \Delta t\sum_{K\in\T_\eta}\m(K)S_K^{k+1}n_K^{k+1} \\
  &\le -\delta\Delta t\sum_{\substack{\sigma\in\E_{\rm int},\\ \sigma=K|L}}
  \tau_\sigma |Dn_{K,\sigma}^{k+1}|^2
  + C\big(\|n_\eta^{k+1}\|_{0,2,\T_{\eta}}^2 + \|S_\eta^{k+1}\|_{0,2,\T_{\eta}}^2\big).
\end{align*}
Inserting this estimate into \eqref{aux2}, we infer that
\begin{align*}
  E^{k+1}-E^k &+ \frac{\Delta t}{2}
  \sum_{\substack{\sigma\in\E_{\rm int},\\ \sigma=K|L}}\tau_\sigma
  \big|D\big(\sqrt{n^{k+1}}\big)_{K,\sigma}\big|^2 
  + \delta\Delta t\sum_{\substack{\sigma\in\E_{\rm int},\\ \sigma=K|L}}
  \tau_\sigma |Dn_{K,\sigma}^{k+1}|^2 \\
  &\le C\big(\|n_\eta^{k+1}\|_{0,2,\T_{\eta}}^2 + \|S_\eta^{k+1}\|_{0,2,\T_{\eta}}^2\big).
\end{align*}
Summing this inequality over $k=0,\ldots,M_T-1$ and observing that the right-hand side is uniformly bounded, we conclude that $(\na^\eta n_\eta)$ is bounded in $L^2(\Omega_T)$, which finishes the proof.
\end{proof}


\section{Convergence of the finite volume scheme}\label{sec.conv}

We prove Theorem \ref{thm.conv}. Consider the family $(n_\eta,S_\eta)_{\eta>0}$ of approximate solutions to \eqref{fv.n0}-\eqref{fv.S}.  In order to apply compactness results, we need to control the difference $n_\eta(\cdot,t+\tau)-n_\eta(\cdot,t)$. To this end, let $\phi\in L^\infty(0,T;H^{2+\varepsilon}(\Omega))$, where $\varepsilon>0$. We denote by $\phi_K$ the average of $\phi$ in the control volume $K$. Using scheme \eqref{fv.n} and the notation $n_\eta^{k+1}:=n_\eta(\cdot,t^{k+1})$,
\begin{align*}
  \sum_{K\in\T_\eta}\m(K)(n_K^{k+1}-n_K^{k})\phi_K
  &\le \frac{\Delta t}{2}\sum_{K\in\T_\eta}\sum_{\sigma\in\E_K}\tau_\sigma
  \big(|Dn_{K,\sigma}^{k+1}| + n_K^{k+1}|DS_{K,\sigma}^{k+1}|\big)|D\phi_{K,\sigma}| \\
  &\le \Delta t\big(|n_\eta^{k+1}|_{1,2,\T_{\eta}}|\phi|_{1,2,\T_{\eta}}
  + \|n_\eta^{k+1}\|_{0,2,\T_{\eta}}|S_\eta^{k+1}|_{1,2,\T_{\eta}}|\phi|_{1,\infty,\T_{\eta}}\big) \\
  &\le C\Delta t\big(|n_\eta^{k+1}|_{1,2,\T_{\eta}} + \|n_\eta^{k+1}\|_{0,2,\T_{\eta}}|S_\eta^{k+1}|_{1,2,\T_{\eta}}
  \big)\|\phi\|_{H^{2+\varepsilon}(\Omega)},
\end{align*}
where $C>0$ only depends on $\Omega$. Summing over $k=0,\ldots,M_T-1$ and employing H\"older's inequality, the uniform bound on $n_\eta$ from Proposition \ref{prop.bound} and on $S_\eta$ from \eqref{nas} imply the existence of a constant $C>0$, independent of $\eta$, such that
\begin{equation}\label{aux6}
  \sum_{k=0}^{M_T-1}\sum_{K\in\T_\eta}\m(K)(n_K^{k+1}-n_K^{k})\phi_K
  \le C\Delta t\|\phi\|_{L^\infty(0,T;H^{2+\varepsilon}(\Omega))}.
\end{equation}
Now, similarly as in the proof of Lemma 10.6 in \cite{EGH00}, for all $0<\tau<\Delta t$,
\begin{align*}
  \int_0^{T-\tau}\int_\Omega & \big(n_\eta(x,t+\tau)-n_\eta(x,t)\big)\phi(x,t)dx\,dt \\
  &\le \sum_{k=0}^{M_T-1}\int_0^{T-\tau}\chi_k(t,t+\tau)dt
  \sum_{K\in\T_{\eta}}\m(K)(n_K^{k+1}-n_K^k)\phi_K,
\end{align*}
where 
$$
  \chi_k(t,t+\tau) = \left\{\begin{array}{ll}
  1 &\quad\mbox{if }k\Delta t\in(t,t+\tau], \\
  0 & \quad\mbox{if }k\Delta t\not\in(t,t+\tau].
  \end{array}\right.
$$
Inserting \eqref{aux6} into the above inequality and observing that
$$
  \int_0^{T-\tau}\chi_k(t,t+\tau)dt \le \tau\le \Delta t,
$$
we infer that
$$
  \int_0^{T-\tau}\int_\Omega \big(n_\eta(x,t+\tau)-n_\eta(x,t)\big)\phi(x,t)dx\,dt
  \le C\|\phi\|_{L^\infty(0,T;H^{2+\varepsilon}(\Omega))}.
$$
This gives a uniform estimate for the time translations of $n_\eta$ in $L^1(0,T;(H^{2+\varepsilon}(\Omega))')$. Since the embedding $H^1(\Omega)\hookrightarrow L^p(\Omega)$ is compact for all $1\le p<\infty$ in two space dimensions, we conclude from the discrete Aubin lemma \cite{DrJu12} that there exists a subsequence of $(n_\eta)$, not relabeled, such that, as $\eta\to 0$,
$$
  n_\eta\to n\quad\mbox{strongly in }L^2(0,T;L^p(\Omega)), \ p<\infty.
$$
Furthermore, since $(\na^\eta n_\eta)$ is bounded in $L^2(\Omega_T)$, there exists $y\in L^2(\Omega_T)$ such that
$$
  \na^\eta n_\eta\rightharpoonup y\quad\mbox{weakly in }L^2(\Omega_T).
$$
It is shown in the proof of Lemma 4.4 in \cite{CLP03} that $y=\na n$ in the sense of distributions. The bound of $(S_\eta)$ in $L^2(0,T;H^1(\Omega))$ implies the existence of a subsequence, which is not relabeled, such that
$$
  S_\eta\rightharpoonup S, \quad \na^\eta S_\eta\rightharpoonup z\quad\mbox{weakly in }
  L^2(\Omega_T).
$$
Again, it follows that $z=\na S$ in the sense of distributions.

The limit $\eta\to 0$ in the scheme \eqref{fv.n0}-\eqref{fv.S} is performed exactly as in the proofs of Propositions 4.2 and 4.3 in \cite{Fil06}, using the above convergence results and the fact that $(n_\eta\na^\eta S_\eta)$ converges weakly to $n\na S$ in $L^1(\Omega_T)$. Compared to \cite{Fil06}, we have to pass to the limit also in the additional cross-diffusion term which does not give any difficulty since this term is linear in $n_\eta$. This shows that $(n,S)$ solves the weak formulation \eqref{weak.n}-\eqref{weak.S}, finishing the proof.


\section{Proof of the discrete logarithmic Sobolev inequality}\label{sec.dlsi}

The proof follows \cite[Lemma 2.1]{DeFe07}. Set 
$$
  v = \frac{\sqrt{m}(u-\bar u)}{\|u-\bar u\|_{0,2,\T}}\in X(\T), 
$$
where $m=\m(\Omega)$ and $\bar u=m^{-1}\int_\Omega udx$.
Then $\int_\Omega vdx=0$ and $m^{-1}\int_\Omega v^2 dx=1$. Using Jensen's inequality for the (probability) measure $m^{-1}v^2 dx$, we find that for $q>2$,
\begin{align*}
  \frac{1}{m}\int_\Omega v^2 \log(v^2)dx
  &= \frac{2}{q-2}\int_\Omega\log(v^{q-2})(m^{-1}v^2 dx)
  \le \frac{2}{q-2}\log\left(\int_\Omega v^{q-2}(m^{-1}v^2dx)\right) \\
  &= \frac{q}{q-2}\log(m^{-1}\|v\|_{0,q,\T}^2) 
  \le \frac{q}{q-2}(m^{-1}\|v\|_{0,q,\T}^2 - 1),
\end{align*}
because of $\log x\le x-1$ for $x>0$. With the definition of $v$, this inequality becomes
$$
  \int_\Omega(u-\bar u)^2\log\frac{(u-\bar u)^2}{m^{-1}\|u-\bar u\|_{0,2,\T}^2}
  dx \le \frac{q}{q-2}\big(\|u-\bar u\|_{0,q,\T}^2 - \|u-\bar u\|_{0,2,\T}^2\big).
$$
By the discrete Sobolev inequality \eqref{dsi}, we infer that for $2<q \leq 2d/(d-2)$ (and $2<q< \infty$ if $d \le 2$)
\begin{align*}
  \int_\Omega(u-\bar u)^2\log\frac{(u-\bar u)^2}{m^{-1}\|u-\bar u\|_{0,2,\T}^2}dx 
  &\le \frac{q}{q-2}\,\frac{C_S(q)^2}{\xi}|u|_{1,2,\T}^2 \\
  &\phantom{xx}{}
  + \frac{q}{q-2}\left(\frac{C_S(q)^2}{\xi}-1\right)\|u-\bar u\|_{0,2,\T}^2.
\end{align*}
Inequality (4.2.19) in \cite{GuZe03} (adapted to domains with general measure) shows that
$$
  \int_\Omega u^2\log\frac{u^2}{m^{-1}\|u\|_{0,2,\T}^2}dx
  \le \int_\Omega(u-\bar u)^2\log\frac{(u-\bar u)^2}{m^{-1}\|u-\bar u\|_{0,2,\T}^2}dx 
  + 2\|u-\bar u\|_{0,2,\T}^2.
$$
Hence, with the discrete Poincar\'e inequality \eqref{dpi},
\begin{align*}
  \int_\Omega u^2\log\frac{u^2}{m^{-1}\|u\|_{0,2,\T}}dx
  &\le \frac{q}{q-2}\,\frac{C_S(q)^2}{\xi}|u|_{1,2,\T}^2 
  + \frac{1}{q-2}\left(\frac{q C_S(q)^2}{\xi}+(q-4)\right)\|u-\bar u\|_{0,2,\T}^2 \\
  &\le \frac{q}{(q-2)\xi}\left(C_S(q)^2 + \frac{C_S(q)^2C_P(2)^2}{\xi} 
  + \frac{q-4}{q}C_P(2)^2\right)|u|_{1,2,\T}^2,
\end{align*}
and Proposition \ref{prop.dlsi} follows.


\section{Long-time behavior}\label{sec.long}

In this section, we prove Theorem \ref{thm.long}. Similarly as in the proof of Proposition \ref{prop.ent} (see \eqref{aux2} and \eqref{aux3}), we have
\begin{align}
  E[n^{k+1}|n^*]&-E[n^k|n^*]
  = \sum_{K\in\T}\m(K)(H(n_K^{k+1})-H(n_K^k)) \nonumber \\
  &\le -\Delta t\sum_{\substack{\sigma\in\E_{\rm int},\\ \sigma=K|L}}
  \tau_\sigma\big|D\big(\sqrt{n^{k+1}}\big)_{K,\sigma}\big|^2
  + \Delta t\sum_{\substack{\sigma\in\E_{\rm int},\\ \sigma=K|L}}
  \tau_\sigma DS_{K,\sigma}^{k+1}Dn_{K,\sigma}^{k+1}. \label{aux7}
\end{align}
In view of the identity $S^*=\mu n^*$, we can formulate the scheme \eqref{fv.S} for all $K\in\T$ as
$$
  0 = \sum_{\sigma\in\E_K}\tau_\sigma D(S^{k+1}-S^*)_{K,\sigma}
  + \delta\sum_{\sigma\in\E_K}\tau_\sigma Dn^{k+1}_{K,\sigma} 
  + \m(K)\big(\mu(n_K^{k+1}-n^*) - (S_K^{k+1}-S^*)\big).
$$
Multiplying this equation by $(S_K^{k+1}-S^*)/\delta$ and summing over $K\in\T$ gives
\begin{align*}
  0 &= 
  -\frac{1}{\delta}\sum_{\substack{\sigma\in\E_{\rm int},\\ \sigma=K|L}}
  \tau_\sigma \big|D(S^{k+1}-S^*)_{K,\sigma}\big|^2
  - \sum_{\substack{\sigma\in\E_{\rm int},\\ \sigma=K|L}}
  \tau_\sigma DS_{K,\sigma}^{k+1}Dn_{K,\sigma}^{k+1} \\
  &\phantom{xx}{}+ \frac{\mu}{\delta}\sum_{K\in\T}\m(K)(n_K^{k+1}-n^*)(S_K^{k+1}-S^*)
  - \frac{1}{\delta}\sum_{K\in\T}\m(K)(S_K^{k+1}-S^*)^2.
\end{align*}
Replacing the last term in \eqref{aux7} by the above equation and using the Cauchy-Schwarz and Young inequalities, it follows that
\begin{align*}
  E[n^{k+1}|n^*]-E[n^k|n^*] &+ \Delta t\big|\sqrt{n^{k+1}}\big|_{1,2,\T}^2
  + \frac{\Delta t}{\delta}\|S^{k+1}-S^*\|_{1,2,\T}^2 \\
  &= \frac{\mu\Delta t}{\delta}\sum_{K\in\T}\m(K)(n_K^{k+1}-n^*)(S_K^{k+1}-S^*) \\
  &\le \frac{\mu^2\Delta t}{2\delta}\|n^{k+1}-n^*\|_{0,2,\T}^2
  + \frac{\Delta t}{2\delta}\|S^{k+1}-S^*\|_{0,2,\T}^2.
\end{align*}
The second term on the right-hand side can be absorbed by the corresponding expression on the left-hand side. For the first term, we employ the continuous embedding of $BV(\Omega)$ into $L^{2}(\Omega)$ in dimension 2 and the definition of the seminorm $|\cdot |_{1,1,\T}$ \cite[Theorem 2]{BCF12}:
$$
  \|n^{k+1}-n^*\|_{0,2,\T}  \le C(\Omega)|n^{k+1}|_{1,1,\T}.
$$
Then, using inequality \eqref{w11}, we can estimate:
\begin{equation}\label{aux8}
  \|n^{k+1}-n^*\|_{0,2,\T}^2 \le \frac{2}{\xi}C(\Omega)^2\|n_0\|_{L^1(\Omega)}
  \big|\sqrt{n^{k+1}}\big|_{1,2,\T}^2.
\end{equation}
Setting $C^*=\mu^2 C(\Omega)^2\|n_0\|_{L^1(\Omega)}/(\delta\xi)$, this yields
$$
  E[n^{k+1}|n^*]-E[n^k|n^*] 
  + \Delta t(1-C^*)\big|\sqrt{n^{k+1}}\big|_{1,2,\T}^2
  + \frac{\Delta t}{2\delta}\|S^{k+1}-S^*\|_{1,2,\T}^2 \le 0.
$$

To proceed, we assume that $C^*<1$. With the discrete logarithmic Sobolev inequality
(Proposition \ref{prop.dlsi}),
$$
  E[n^{k+1}|n^*] \le C_L|\sqrt{n^{k+1}}|_{1,2,\T}^2,
$$
we infer that 
\begin{equation}\label{aux9}
  E[n^{k+1}|n^*]\left(1+\frac{1-C^*}{C_L}\Delta t\right)-E[n^k|n^*] 
  + \frac{\Delta t}{2\delta}\|S^{k+1}-S^*\|_{1,2,\T}^2 \le 0,
\end{equation}
and hence,
\begin{equation}\label{decay.e}
  E[n^k|n^*] \le \left(1+\frac{1-C^*}{C_L}\Delta t\right)^{-k} E[n^0|n^*].
\end{equation}
Then, by the Csisz\'ar-Kullback inequality \cite[Prop.~3.1]{CCD02}
(this result is valid in bounded domains too),
$$
  \|n^k-n^*\|_{0,1,\T}^2 \le 4\|n_0\|_{L^1(\Omega)}E[n^k|n^*] 
  \le 4\|n_0\|_{L^1(\Omega)}
  \left(1+\frac{1-C^*}{C_L}\Delta t\right)^{-k} E[n^0|n^*].
$$
Going back to \eqref{aux9}, we find that
$$
  \|S^{k+1}-S^*\|_{1,2,\T}^2
  \le \frac{2\delta}{\Delta t}E[n^k|n^*] 
  \le \frac{2\delta}{\Delta t}\left(1+\frac{1-C^*}{C_L}\Delta t\right)^{-k} E[n^0|n^*],
$$
which concludes the proof.



\section{Numerical experiments}\label{sec.num}

In this section, we investigate the numerical convergence rates and give some examples illuminating the long-time behavior of the finite volume solutions to nonhomogeneous steady states.

\subsection{Numerical convergence rates}

We compute first the spatial convergence rate of the numerical scheme. We consider the system \eqref{ks} on the square $\Omega=(-\frac12,\frac12)^2$. The time step is chosen to be $\Delta t=10^{-8}$, and the final time is  $t=10^{-4}$. The initial data is the Gaussian
$$
  n_0(x,y) = \frac{M}{2\pi \theta}\exp\left(-\frac{(x-x_0)^2+(y-y_0)^2}{2\theta}\right),
$$
where $\theta=10^{-2}$, $M=\|n_0\|_{L^1(\Omega)}=6\pi$, and $x_0=y_0=0.1$. The model parameters are $\delta=10^{-3}$ and $\mu=1$. We compute the numerical solution on a sequence of square meshes. The coarsest mesh is composed of $4\times 4$ squares. The sequence of meshes is obtained by dividing successively the size of the squares by 4. Then, the finest grid is made of $256\times 256$ squares. The $L^p$ error at time $t$ is given by 
$$
  e_{\Delta x} = \|n_{\Delta x}(\cdot,t) - n_{\rm ex}(\cdot,t)\|_{L^p(\Omega)},
$$
where $n_{\Delta x}$ represents the approximation of the cell density computed from a mesh of size $\Delta x$ and $n_{\rm ex}$ is the ``exact'' solution computed from a mesh with $256\times 256$ squares (and with $\Delta t=10^{-8}$). The numerical scheme is said to be of order $m$ if for all sufficiently small $\Delta x>0$, it holds that $e_{\Delta x}\le C(\Delta x)^m$ for some constant $C>0$. Figure \ref{fig.cr} shows that the convergence rates in the $L^1$, $L^2$, and $L^\infty$ norms are around one. As expected, the scheme is of first order.

\begin{figure}[ht]
\centering
\includegraphics[scale=0.55]{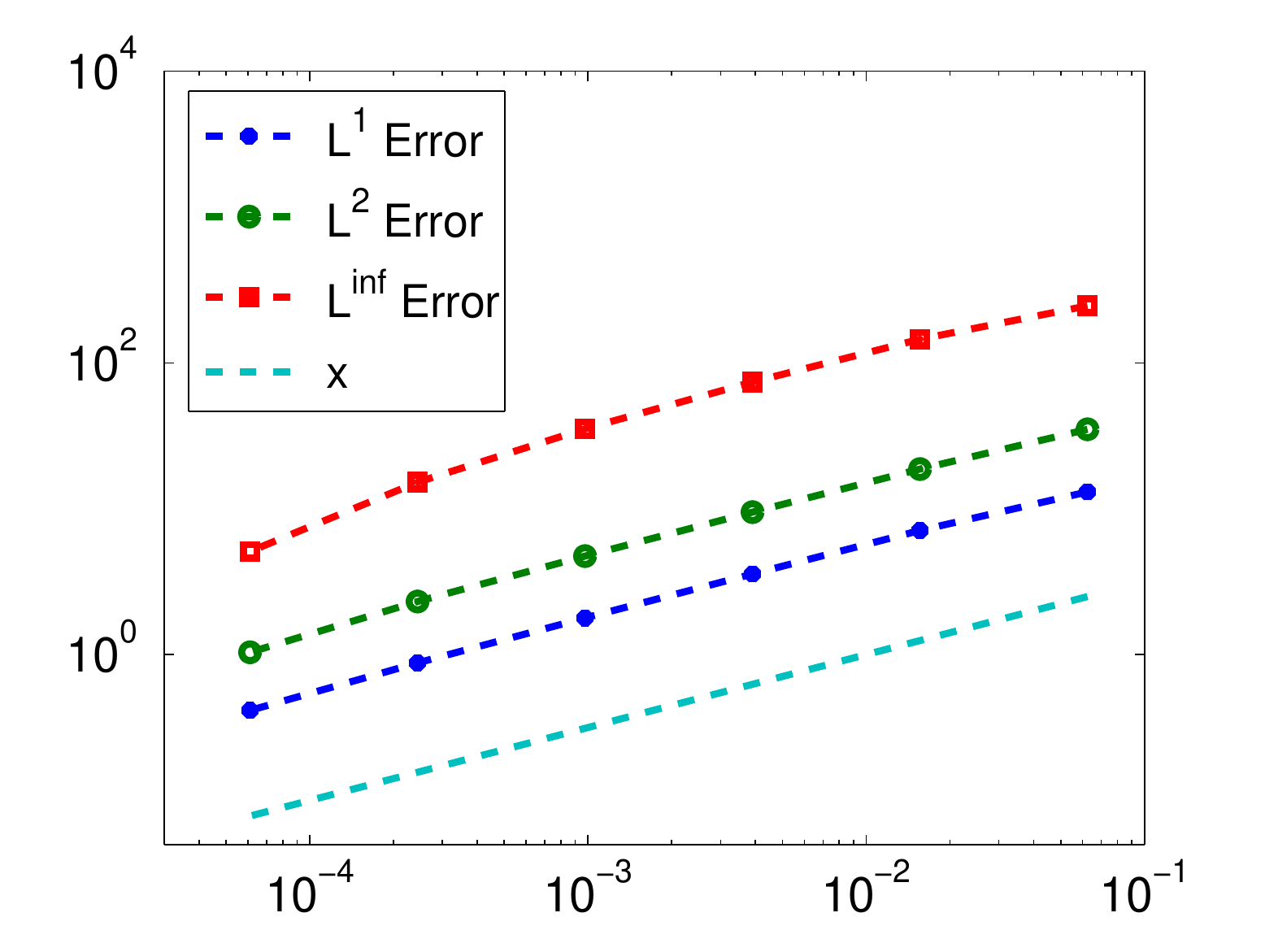}
\caption{Spatial convergence orders in the $L^{1}$, $L^{2}$, and $L^{\infty}$ norm.}
\label{fig.cr}
\end{figure}


\subsection{Decay rates}

According to Theorem \ref{thm.long}, the solution to the Keller-Segel system converges to the homogeneous steady state if $\mu$ or $1/\delta$ are sufficiently small. We will verify this property experimentally. To this end, let $\Omega=(-\frac12,\frac12)^2$ and
$$
  n_0(x,y) = \frac{M}{2\pi\theta}\exp\left(-\frac{x^2+y^2}{2\theta}\right),
$$
where $\theta=10^{-2}$ and $M=5\pi$. We compute the approximate solution on a $32\times 32$ Cartesian grid, and
we choose $\Delta t=2\cdot 10^{-4}$. In Figure \ref{fig.ent1}, we depict the temporal evolution of the relative entropy $E^*(t^k)=E[n^k|n^*]$ in semi-logarithmic scale. In all cases shown, the convergence seems to be of exponential rate. The rate becomes larger for larger values of $\delta$ or smaller values of $\mu$ which is in agreement
with estimate \eqref{decay.e}. In fact, the constant $C^*$ is proportional to
$\mu^2/\delta$ (see Theorem \ref{thm.long}) and the rate improves if
$\mu^2/\delta$ is smaller.

\begin{figure}[ht]
\centering
\subfigure[$\mu = 1$.]{\includegraphics[width=75mm]{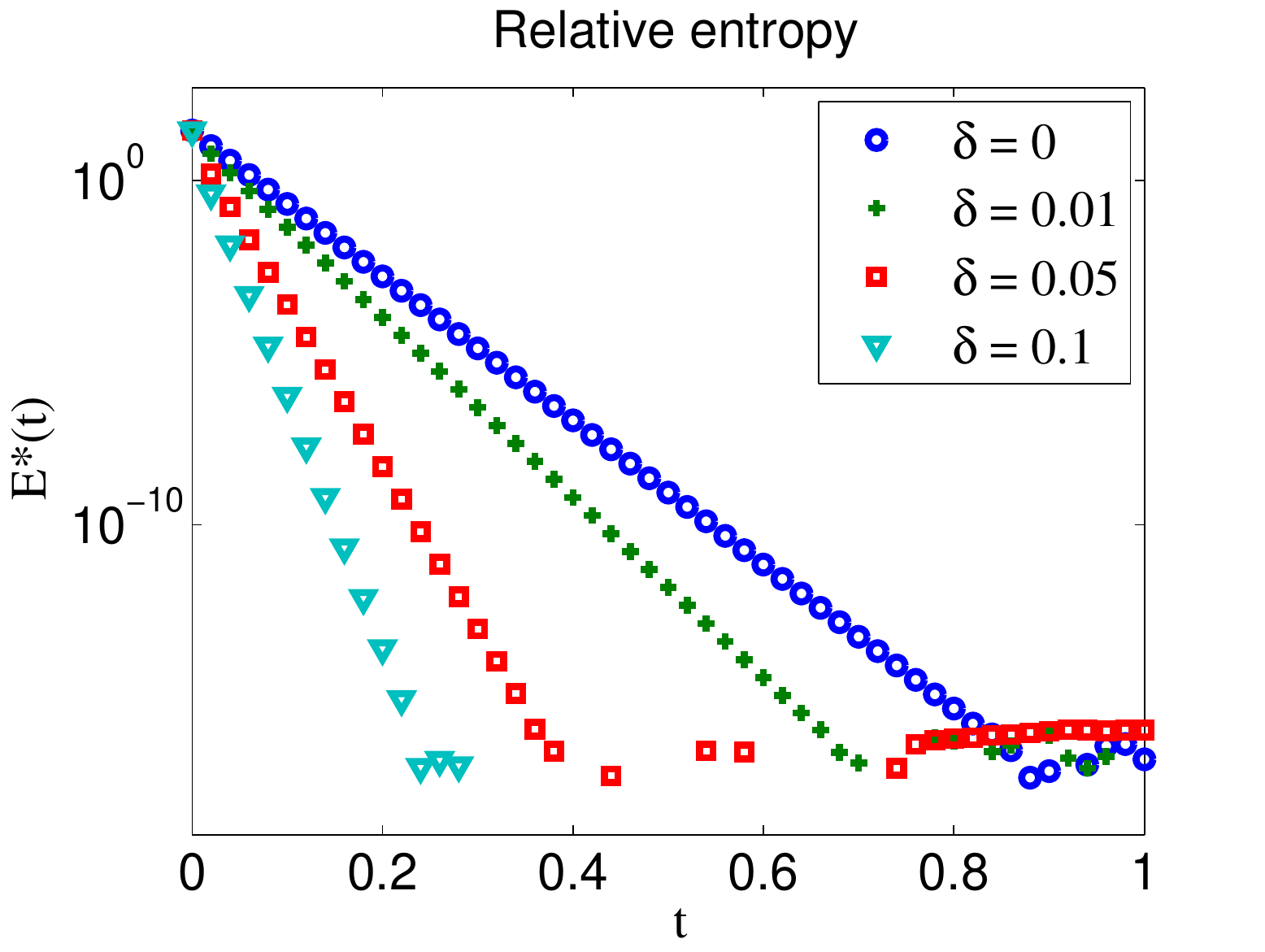}\label{fig1}}
\subfigure[$\delta = 10^{-3}$.]{\includegraphics[width=75mm]{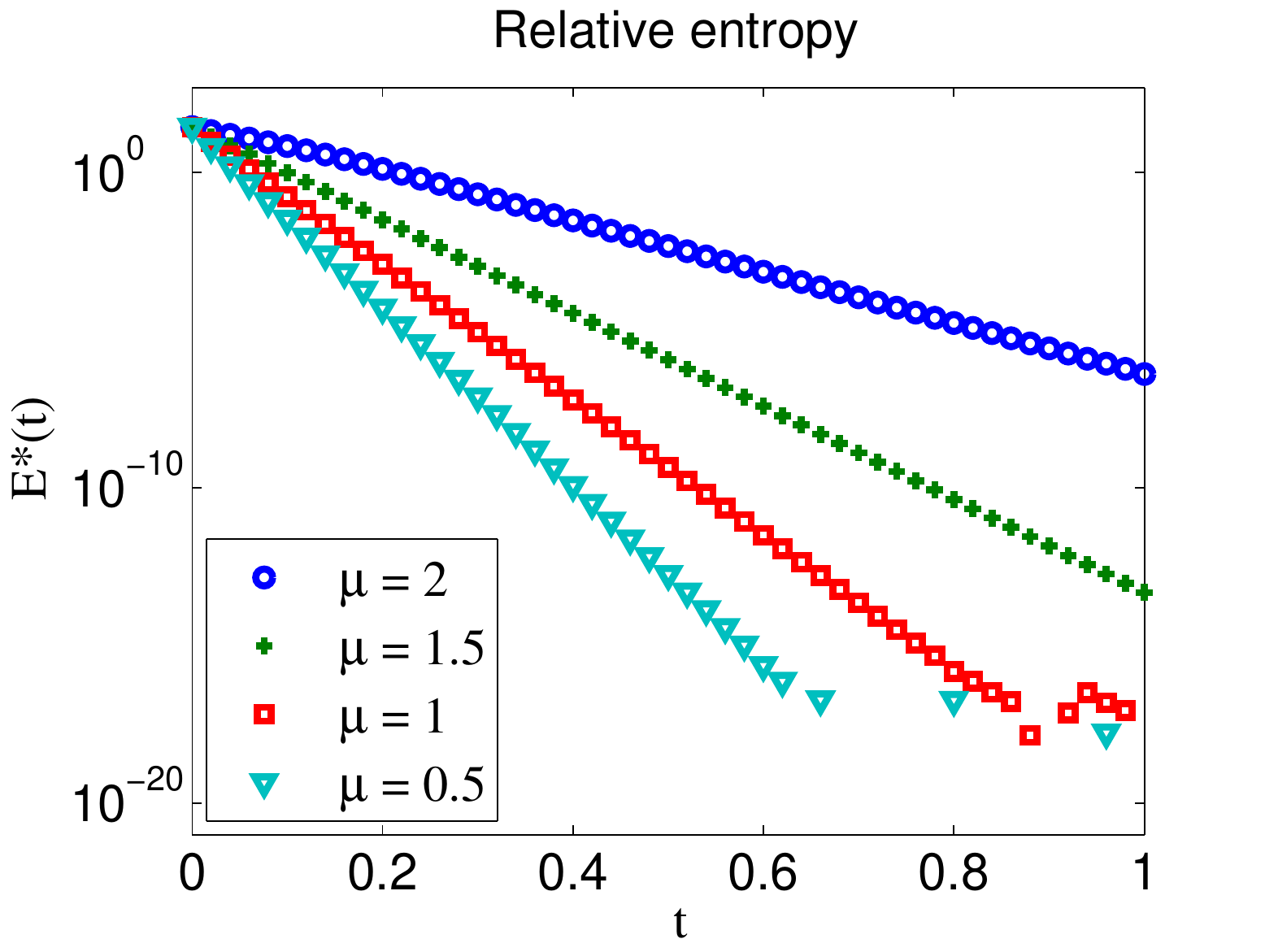}\label{fig2}}
\caption{Relative entropy $E[n^k|n^*]$ versus time $t^k$ in semi-logarithmic scale
for various values of $\delta$ and $\mu$.}
\label{fig.ent1}
\end{figure}

As a numerical check, we computed the evolution of the relative entropies for different grid sizes $N$ and different time step sizes $\Delta t$. Figure \ref{fig.ent2} shows that the decay rate does not depend on the time step or the mesh considered.

\begin{figure}[ht]
\centering
\subfigure[$\Delta t=2.10^{-5}$.]{\includegraphics[width=75mm]{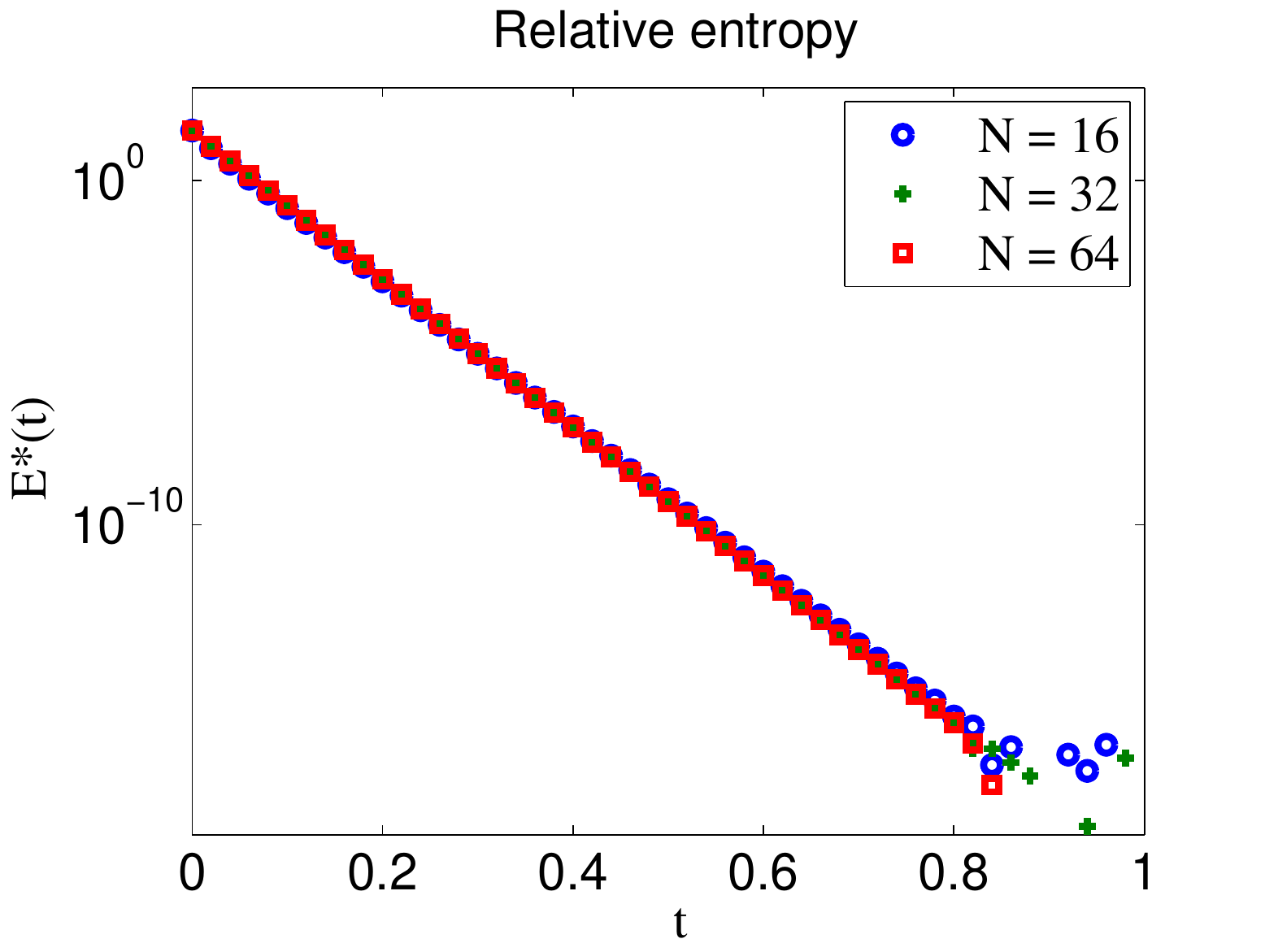}}
\subfigure[$N=16$.]{\includegraphics[width=75mm]{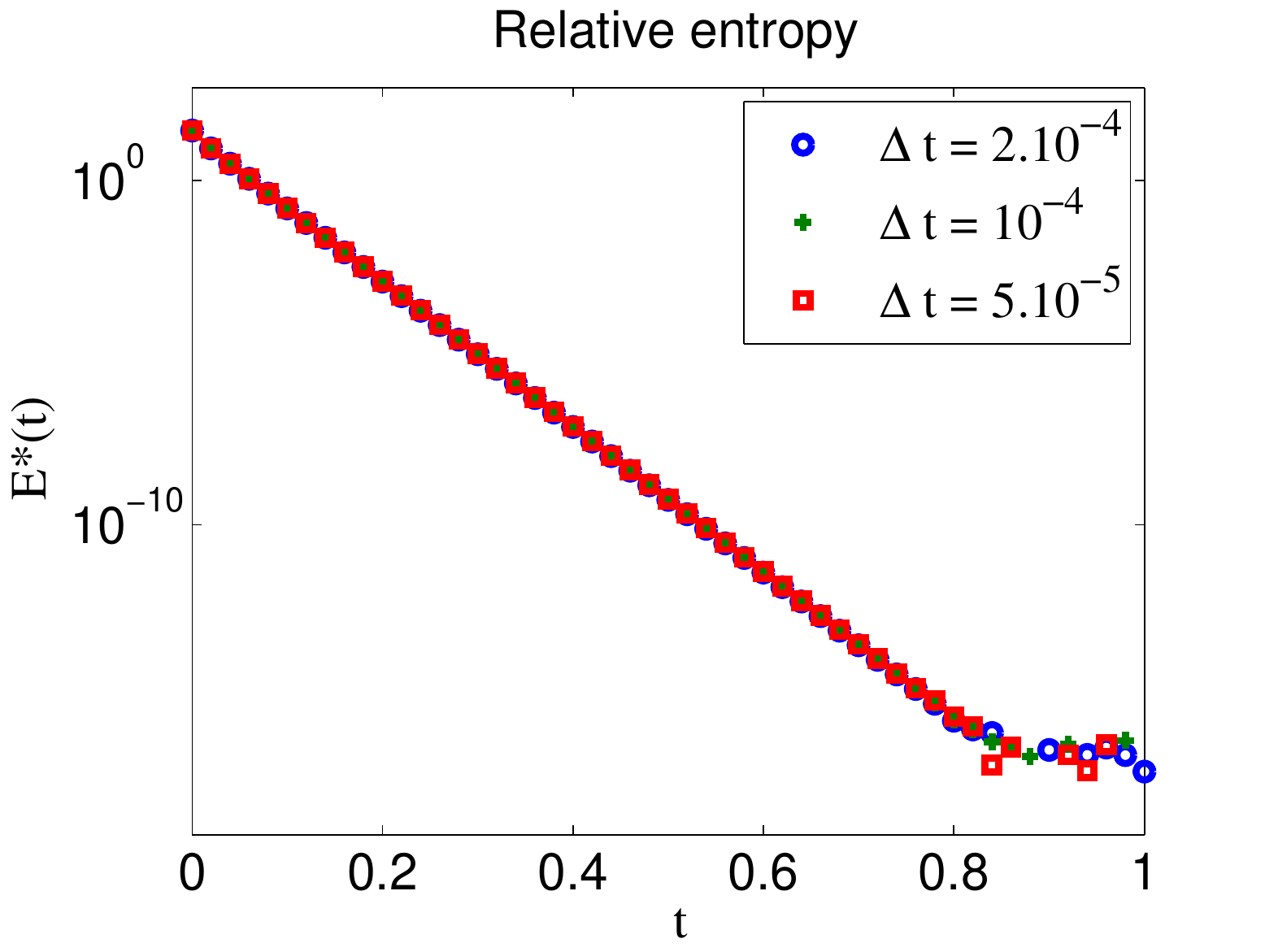}}
\caption{Relative entropy $E[n^k|n^*]$ versus time $t^k$ in semi-logarithmic scale for various mesh and time step sizes.}
\label{fig.ent2}
\end{figure}


\subsection{Nonsymmetric initial data on a square}

In this subsection, we explore the behavior of the solutions to \eqref{ks} for different values of $\delta$. We choose $\Omega=(-\frac12,\frac12)^2$ with a $64\times 64$ Cartesian grid, $\mu=1$, and $\Delta t=2\cdot 10^{-5}$. We consider two nonsymmetric initial functions with mass $6\pi$:
\begin{align}
  n_{0,1}(x,y) &= \frac{6\pi}{2\pi \theta}\exp\left(-\frac{(x-x_0)^2+(y-y_0)^2}{2\theta}
  \right), \label{n01} \\
  n_{0,2}(x,y) &= \frac{4\pi}{2\pi \theta}\exp\left(-\frac{(x-x_0)^2+(y-y_0)^2}{2\theta}
  \right) + \frac{2\pi}{2\pi \theta}\exp\left(-\frac{(x-x_1)^2+(y-y_1)^2}{2\theta}
  \right), \label{n02}
\end{align}
where $\theta=10^{-2}$, $x_0=y_0=0.1$, and $x_1=y_1=-0.2$ (see Figure \ref{fig.init}).

\begin{figure}[ht]
\centering
\subfigure[Initial datum $n_{0,1}$.]{\includegraphics[width=70mm]{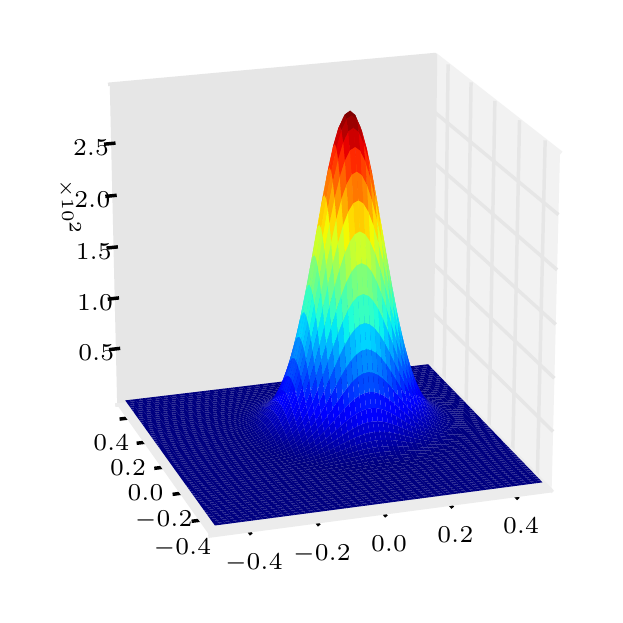}}
\subfigure[Initial datum $n_{0,2}$.]{\includegraphics[width=70mm]{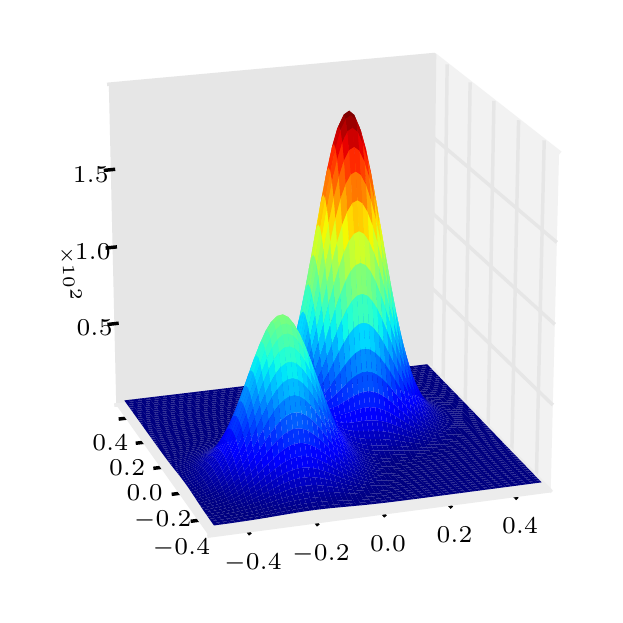}}
\caption{Initial cell densities.}
\label{fig.init}
\end{figure}

We consider first the case $\delta=0$, which corresponds to the classical parabolic-elliptic Keller-Segel system. In this case, our finite volume scheme coincides with that of \cite{Fil06}. We recall that solutions to the classical parabolic-elliptic model blow up in finite time if the initial mass satisfies $M>4\pi$ \cite{Nag01} (in the non-radial case). The numerical results  at a time just before the numerical blow-up are presented in Figure \ref{fig.nonsymm.0}. We observe the blow-up of the cell density in finite time, and the blow-up occurs at the boundary, as expected. More precisely, it occurs at that corner which is closest to the global maximum of the initial datum.

\begin{figure}[ht]
\centering
\subfigure[Initial datum $n_{0,1}$, $t=1$.]{\includegraphics[width=70mm]{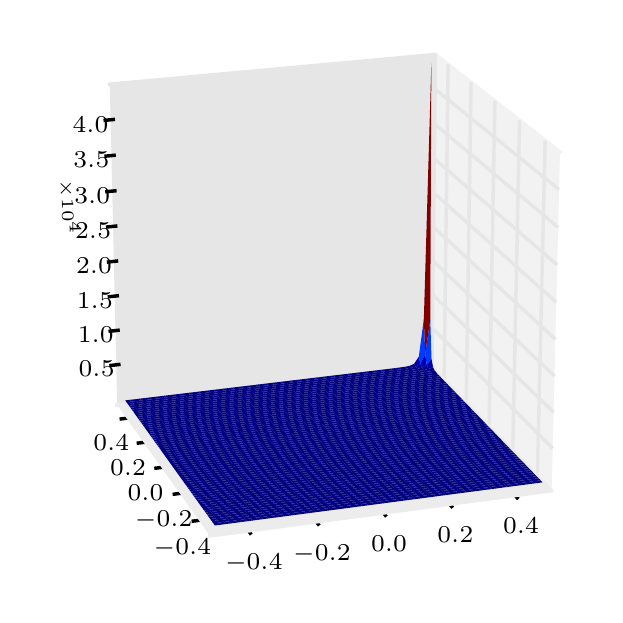}}
\subfigure[Initial datum $n_{0,2}$, $t=0.6$.]{\includegraphics[width=70mm]{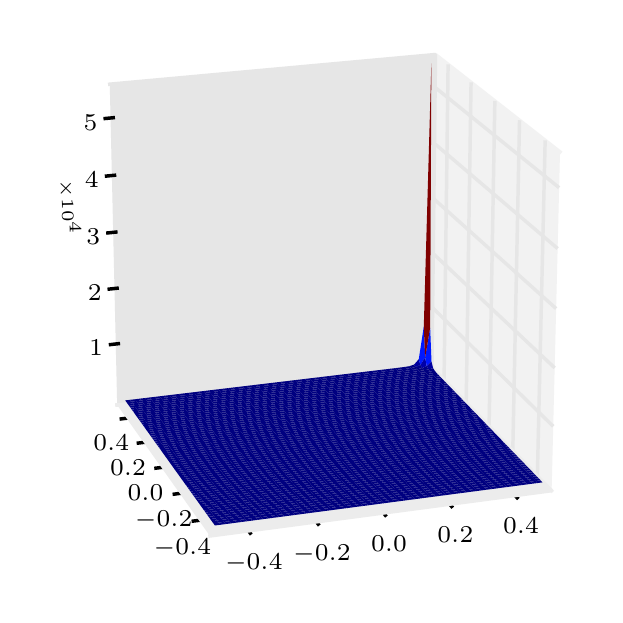}}
\caption{Cell density computed from nonsymmetric initial data with $M=6\pi$ and $\delta=0$.}
\label{fig.nonsymm.0}
\end{figure}

Next, we choose $\delta=10^{-3}$ and $\delta=10^{-2}$. According to Theorem \ref{thm.ex}, the numerical solution exists for all time. This behavior is confirmed in Figure \ref{fig.nonsymm.23}, where we show the cell density at time $t=5$. At this time, the solution is very close to the steady state which is nonhomogeneous. We observe a smoothing effect of the cross-diffusion parameter $\delta$; the cell density maximum decreases with increasing values of $\delta$.

\begin{figure}[ht]
\centering
\subfigure[Initial datum $n_{0,1}$, $\delta=10^{-3}$.]{\includegraphics[width=70mm]{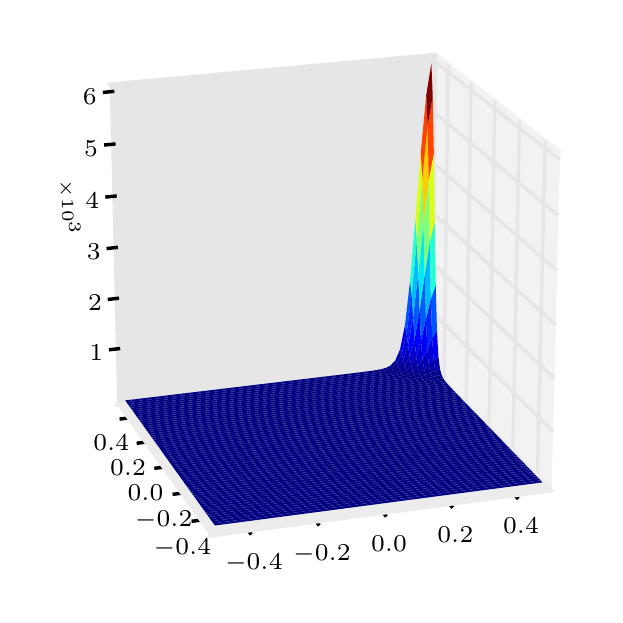}}
\subfigure[Initial datum $n_{0,2}$, $\delta=10^{-3}$.]{\includegraphics[width=70mm]{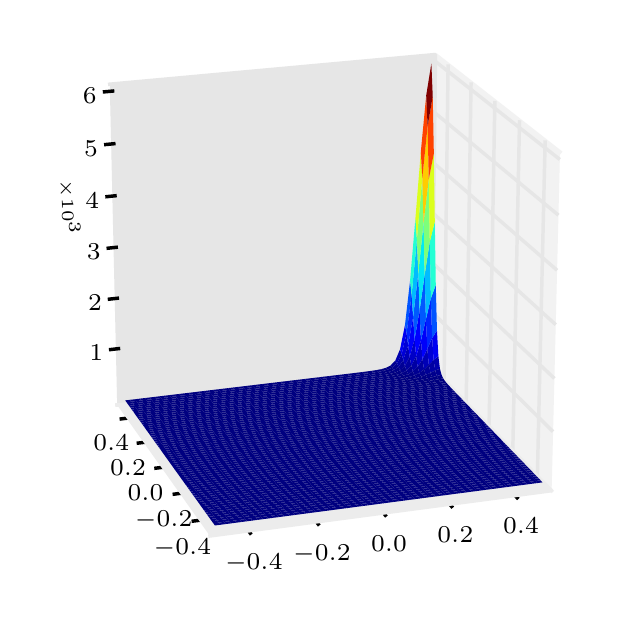}}
\subfigure[Initial datum $n_{0,1}$, $\delta=10^{-2}$.]{\includegraphics[width=70mm]{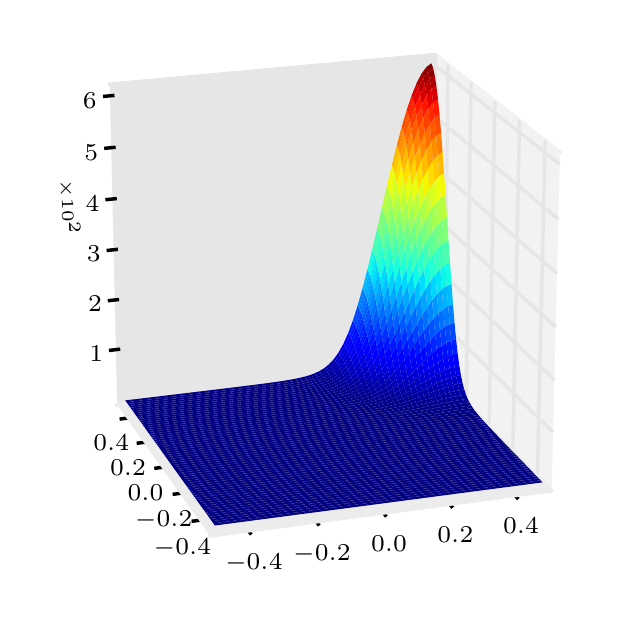}}
\subfigure[Initial datum $n_{0,2}$, $\delta=10^{-2}$.]{\includegraphics[width=70mm]{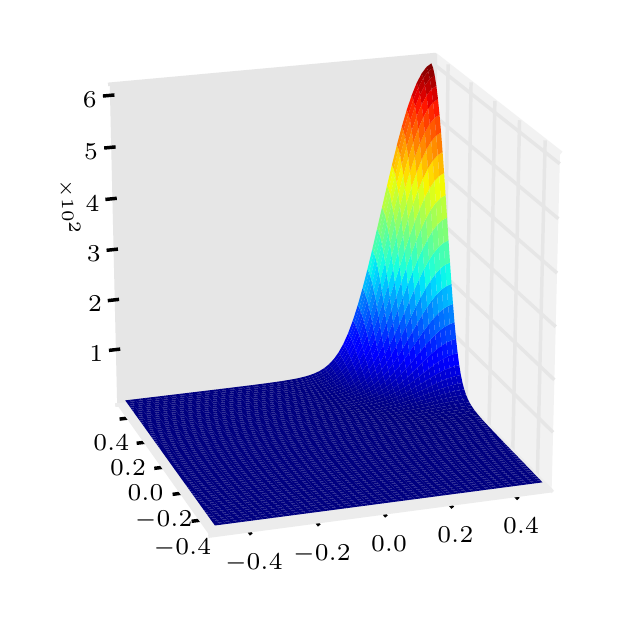}}
\caption{Cell density computed at $t=5$ from nonsymmetric initial data with $M=6\pi$ for different values of $\delta$.}
\label{fig.nonsymm.23}
\end{figure}


\subsection{Symmetric initial data on a square}

We consider, as in the previous subsection, the domain $\Omega=(-\frac12,\frac12)^2$ with a $64\times 64$ Cartesian grid, $\mu=1$, and $\Delta t=10^{-5}$. Here, we consider the radially symmetric initial datum
\begin{equation}\label{n03}
  n_{0,3}(x,y) = \frac{M}{2\pi\theta}\exp\left(-\frac{x^2+y^2}{2\theta}\right)
\end{equation}
with $M=20\pi$ and $\theta=10^{-2}$. Since $M>8\pi$ and the initial datum is radially symmetric, we expect that the solution to the classical Keller-Segel model ($\delta=0$) blows up in finite time \cite{Nag95,Per05}. Figure \ref{fig.symm.0} shows that this is indeed the case, and blow-up occurs in the center of the domain. 

\begin{figure}[ht]
\centering
\includegraphics[width=70mm]{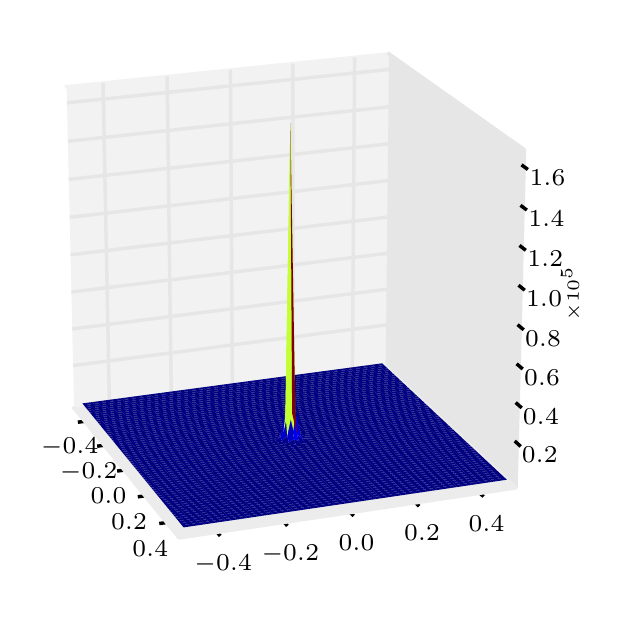}
\caption{Cell density at time $t=0.05$ computed from to the radially symmetric initial datum $n_{0,3}$ with $M=20 \pi$ and $\delta=0$.}
\label{fig.symm.0}
\end{figure}

In contrast to the classical Keller-Segel system, when taking $\delta=10^{-3}$, the cell density peak moves to a corner of the domain and converges to a nonhomogeneous steady state (see Figure \ref{fig.symm.3}). The time evolution of the $L^\infty$ norm of the cell density shows an interesting behavior (see Figure \ref{fig.symm.Linfty}). We observe two distinct levels. The first one is reached almost instantaneously. The $L^\infty$ norm stays almost constant and the cell density seems to stabilize at an intermediate symmetric state (Figure \ref{fig.symm.3}a). After some time, the $L^\infty$ norm increases sharply and the cell density peak moves to the boundary (Figure \ref{fig.symm.3}b). Then the solution stabilizes again (Figure \ref{fig.symm.3}c). We note that we obtain the same steady state when using a Gaussian centered at $(-10^{-3},-10^{-3})$.

\begin{figure}[ht]
\centering
\subfigure[$t=0.6$.]{\includegraphics[width=53mm]{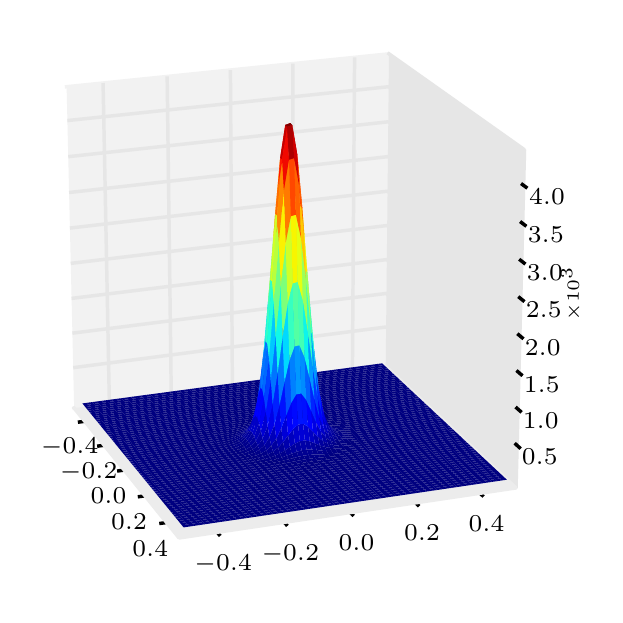}\label{subfigure7_1}}
\subfigure[$t=0.73$.]{\includegraphics[width=53mm]{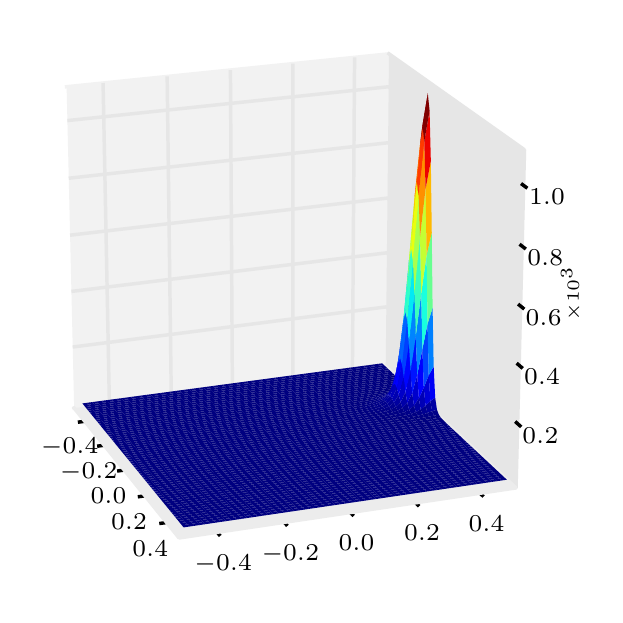}\label{subfigure7_2}}
\subfigure[$t=5$.]{\includegraphics[width=53mm]{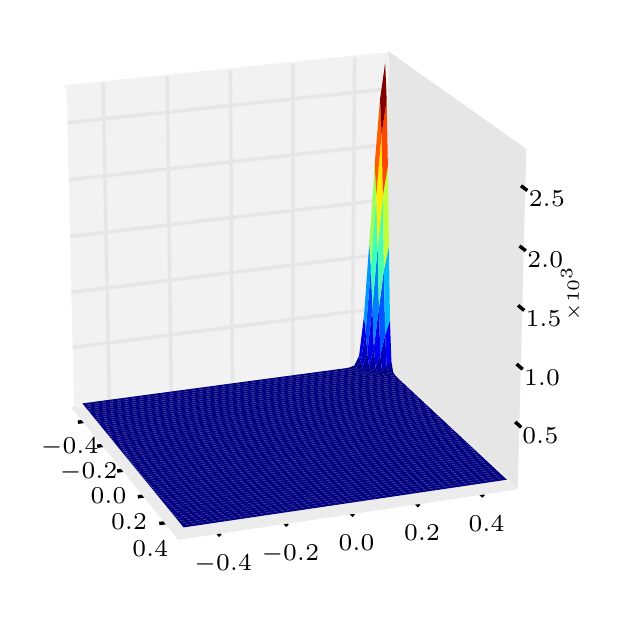}\label{subfigure7_3}}
\caption{Cell density computed from the radially symmetric initial datum $n_{0,3}$ with $M=20\pi$ and $\delta=10^{-3}$.}
\label{fig.symm.3}
\end{figure}

\begin{figure}[ht]
\centering
\includegraphics[width=70mm]{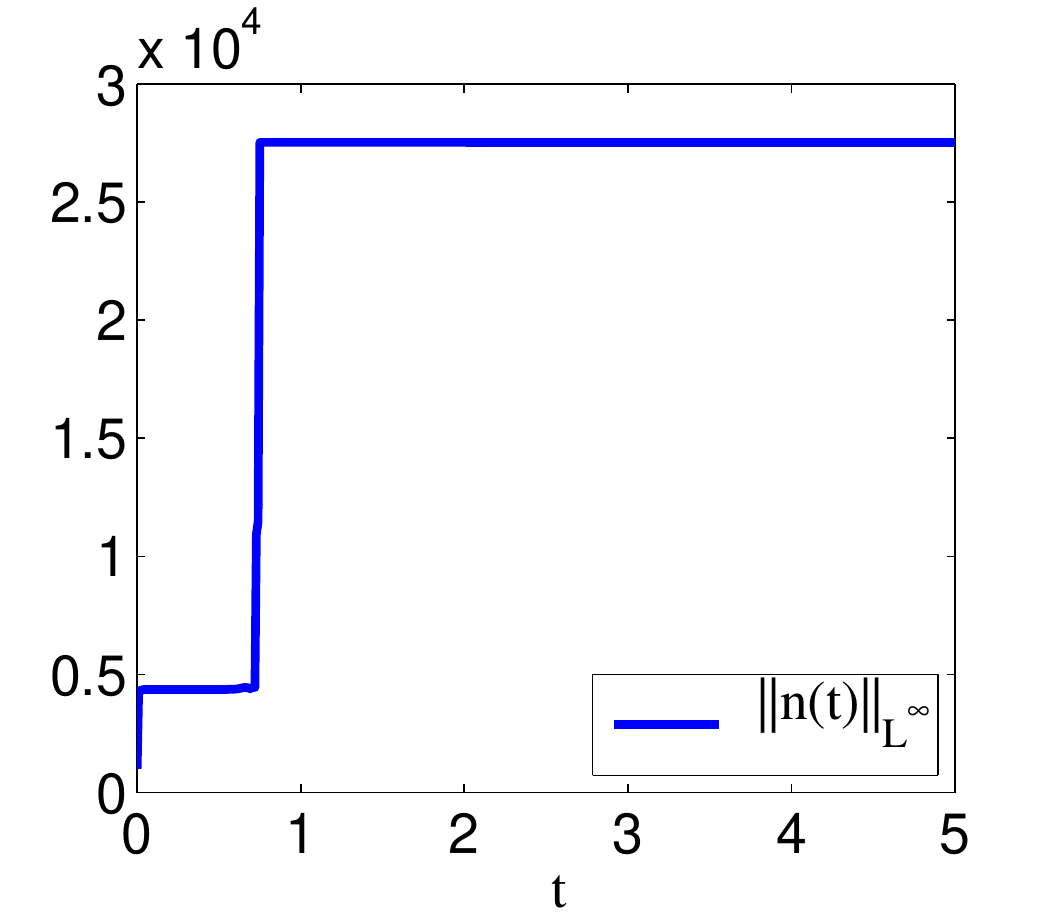}
\caption{Time evolution of $\|n^k\|_{L^\infty(\Omega)}$ computed from the radially symmetric initial datum $n_{0,3}$ with $M=20\pi$ and $\delta=10^{-3}$.}
\label{fig.symm.Linfty}
\end{figure}


\subsection{Nonsymmetric initial data on a rectangle}

We consider the domain $\Omega=(-1,1)$ $\times(-\frac12,\frac12)$ and compute the approximate solutions on a $128\times 64$ Cartesian grid with $\Delta t=5\cdot 10^{-5}$. The secretion rate is again $\mu=1$, and we choose the initial data $n_{0,1}$ and $n_{0,2}$, defined in \eqref{n01}-\eqref{n02} with mass $M=6\pi$. If $\delta=0$, the solution blows up in finite time and the blow up occurs in a corner as in the square domain (see Figure \ref{fig.nonsymm.r0}). If $\delta=10^{-3}$, the approximate solutions converge to a non-homogeneous steady state (Figure \ref{fig.nonsymm.r3}). Interestingly, before moving to the corner, the solution evolving from the nonsymmetric initial datum $n_{0,2}$ shows some intermediate behavior; see Figure \ref{fig.nonsymm.r3}b.

\begin{figure}[ht]
\centering
\subfigure[Initial datum $n_{0,1}$, $t=0.5$.]{\includegraphics[width=70mm]{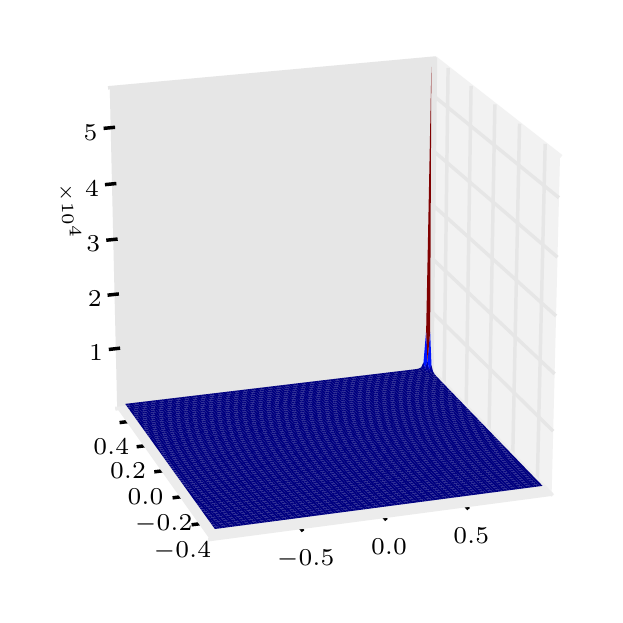}}
\subfigure[Initial datum $n_{0,2}$, $t=1.7$.]{\includegraphics[width=70mm]{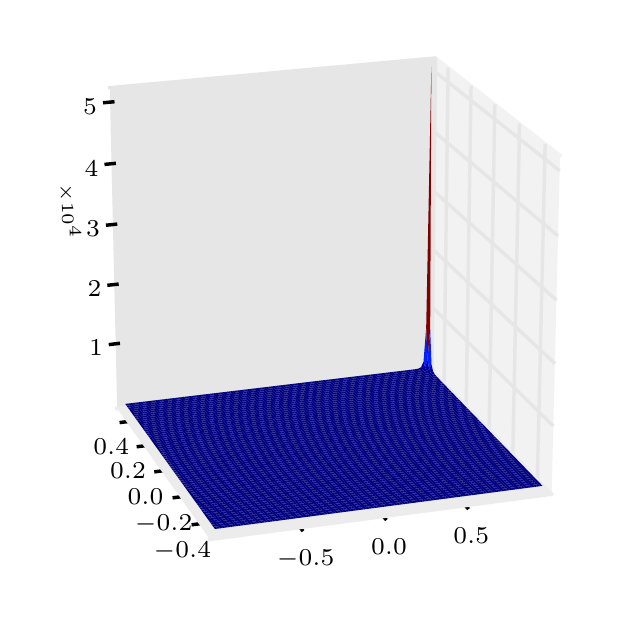}}
\caption{Cell density computed from nonsymmetric initial data with $M=6\pi$ and $\delta=0$.}
\label{fig.nonsymm.r0}
\end{figure}

\begin{figure}[ht]
\centering
\subfigure[Initial datum $n_{0,1}$, $t=1$.]{\includegraphics[width=70mm]{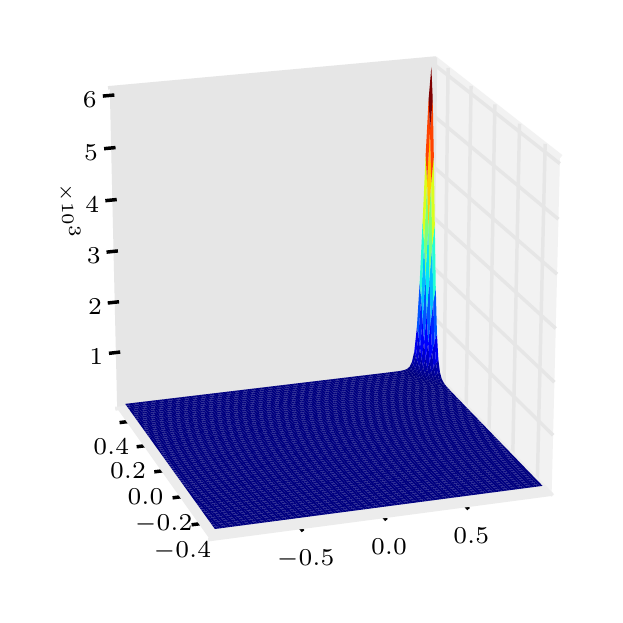}}
\subfigure[Initial datum $n_{0,2}$, $t=1$.]{\includegraphics[width=70mm]{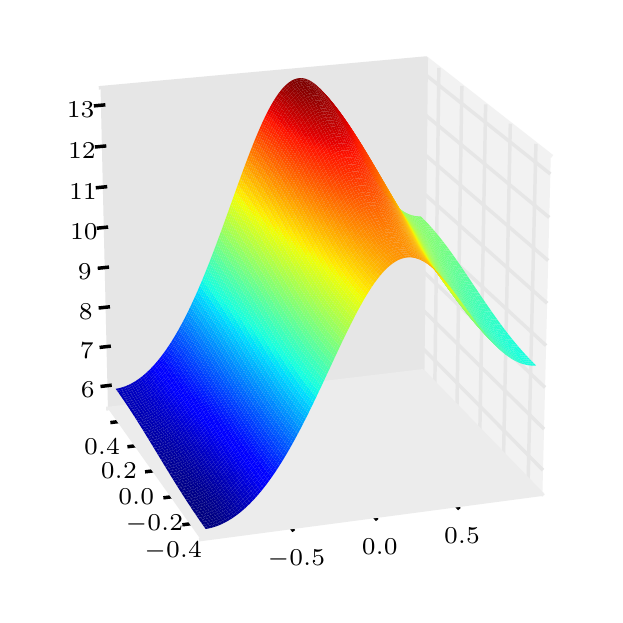}}
\subfigure[Initial datum $n_{0,1}$, $t=5$.]{\includegraphics[width=70mm]{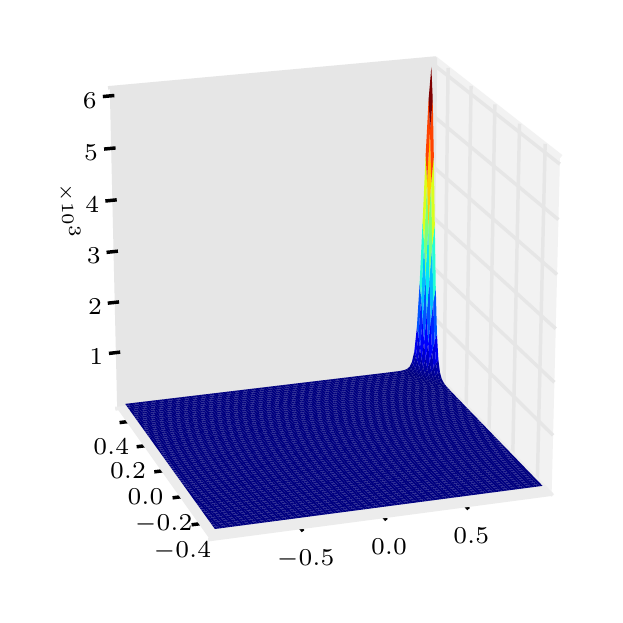}}
\subfigure[Initial datum $n_{0,2}$, $t=5$.]{\includegraphics[width=70mm]{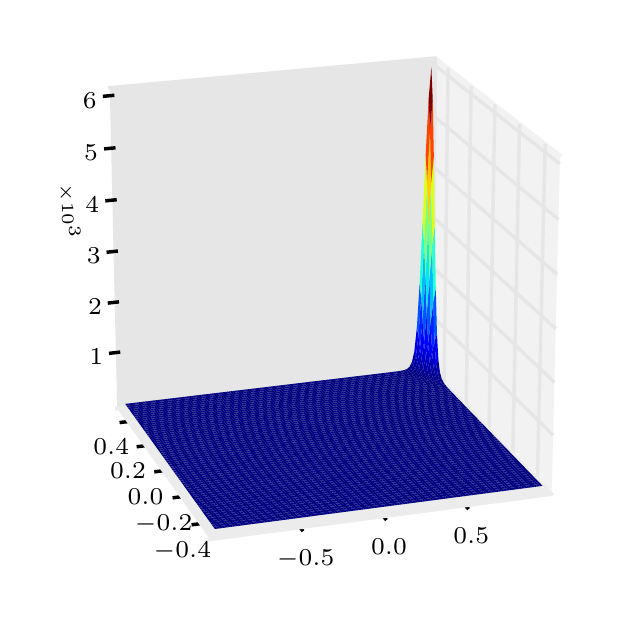}}
\caption{Cell density computed from nonsymmetric initial data with $M=6\pi$ and $\delta=10^{-3}$.}
\label{fig.nonsymm.r3}
\end{figure}


\subsection{Symmetric initial data on a rectangle}

The domain is still the rectangle $\Omega=(-1,1)\times(-\frac12,\frac12)$, we take a $128\times 64$ Cartesian grid, $\mu=1$, and $\Delta t=10^{-5}$. We choose the initial datum $n_{0,3}$, defined in \eqref{n03}, with $M=20\pi$. Clearly, the approximate solution to the classical Keller-Segel model $\delta=0$ blows up in finite time in the center $(0,0)$ of the rectangle. When $\delta=10^{-3}$, the cell density peak first moves to the closest boundary point before moving to a corner of the domain, as in the square domain (Figure \ref{fig.symm.r}). However, in contrast to the case of a square domain, there exist {\em two} intermediate states, one up to time $t\approx 0.9$ and another in the interval $(0.9,2.3)$, and one final state for long times (see Figure \ref{fig.symm.rL}). We note that the same qualitative behavior is obtained using $\delta=10^{-2}$.

\begin{figure}[ht]
\centering
\subfigure[$t=0.7$.]{\includegraphics[width=53mm]{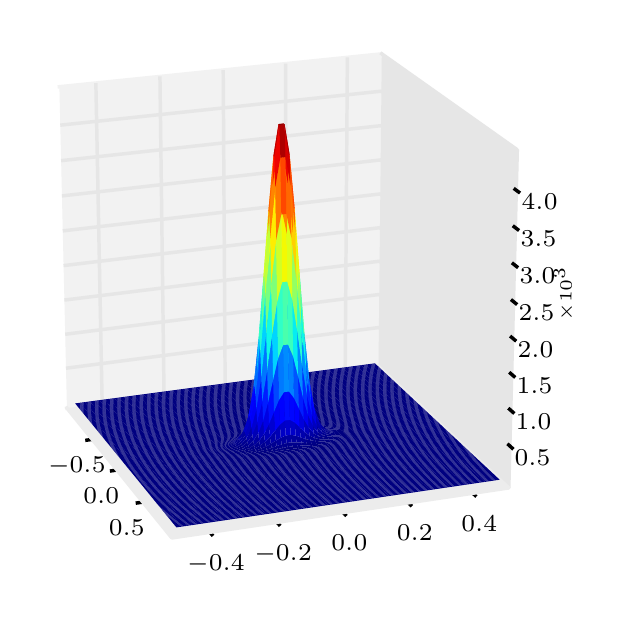}\label{subfigure16_1}}
\subfigure[$t=1.9$.]{\includegraphics[width=53mm]{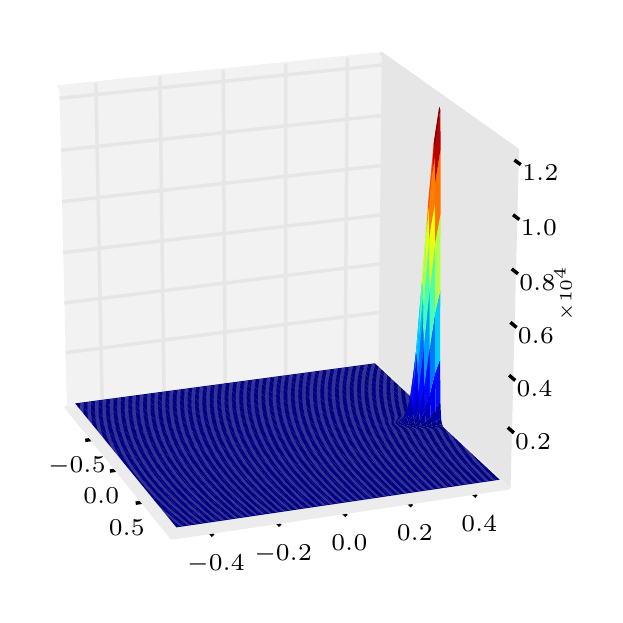}\label{subfigure16_2}}
\subfigure[$t=2.5$.]{\includegraphics[width=53mm]{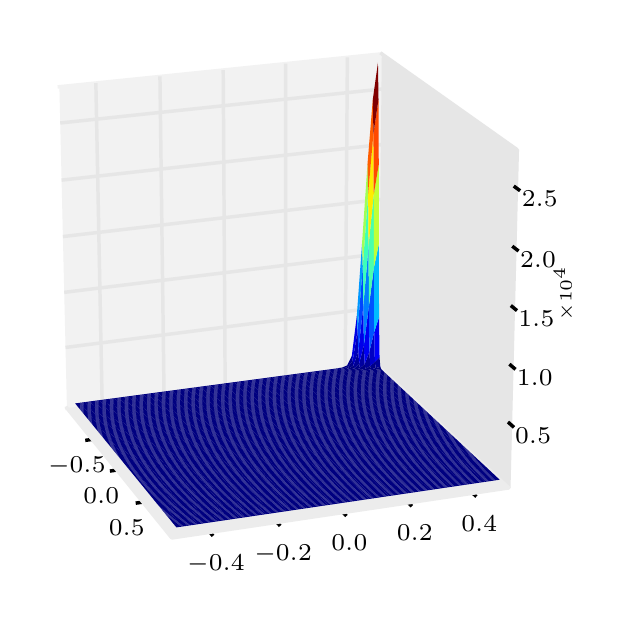}\label{subfigure16_3}}
\caption{Cell density computed from the symmetric initial datum $n_{0,3}$ with $M=20\pi$ and $\delta=10^{-3}$.}
\label{fig.symm.r}
\end{figure}

\begin{figure}[ht]
\centering
\includegraphics[width=70mm]{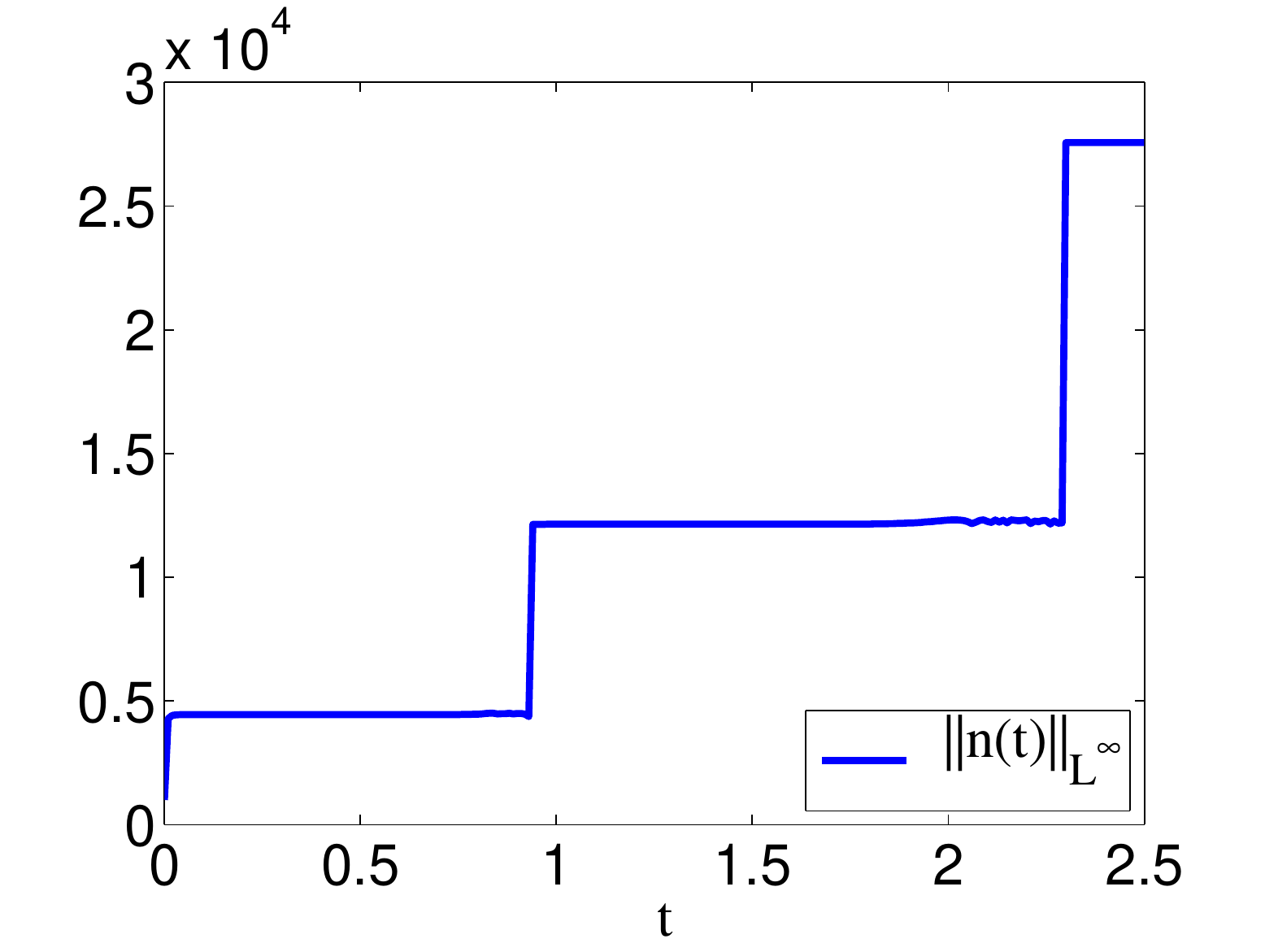}
\caption{Time evolution of $\|n^k\|_{L^\infty(\Omega)}$ computed from the radially symmetric initial datum $n_{0,3}$ with $M=20\pi$ and $\delta=10^{-3}$.}
\label{fig.symm.rL}
\end{figure}


\end{document}